\documentclass[10pt]{amsart}
\usepackage{amsmath}
\usepackage{amsxtra}
\usepackage{amscd}
\usepackage{amsthm}
\usepackage{amsfonts}
\usepackage{amssymb}
\usepackage{eucal}
\usepackage{mathabx}
\usepackage{xcolor}
\usepackage[matrix,arrow,curve]{xy}
\usepackage{pb-diagram}
\usepackage{comment}
\usepackage{array}
\newcommand{\fg}{{\mathfrak{g}}}
\newcommand{\bp}{{\mathbb{P}}}
\newcommand{\fh}{{\mathfrak{h}}}

\newcommand{\fz}{{\mathfrak{z}}}
\newcommand{\fsl}{{\mathfrak{sl}}}

\DeclareMathOperator{\Spec}{Spec}
\DeclareMathOperator{\Hom}{Hom}
\newcommand{\fn}{{\mathfrak{n}}}
\newcommand{\cg}{\mathcal{G}}

\newcommand{\End}{{\rm End} \,}
\newcommand{\Id}{{\rm Id} \,}

\newcommand{\rk}{{\rm rk} \,}
\newcommand{\gr}{{\rm Gr} \,}
\newcommand{\ad}{{\rm ad} \,}

\newcommand{\bc}{\mathbb{C}}
\newcommand{\bz}{\mathbb{Z}}

\newcommand{\diag}{{\rm diag} \,}

\newcommand{\ft}{\mathfrak{t}}
\newcommand{\bq}{\mathbb{Q}}
\newcommand{\mS}{\mathcal{S}}

\newtheorem{defn}[subsection]{Definition}
\newtheorem{thm}[subsection]{Theorem}
\newtheorem*{thm*}{Theorem}
\newtheorem{lem}[subsection]{Lemma}
\newtheorem{prop}[subsection]{Proposition}

\newtheorem{cor}[subsection]{Corollary}

\newtheorem*{rem}{Remark}

\newtheorem*{example}{Example}

\DeclareMathOperator{\spann}{span}

\begin{document}

\title{Bethe subspaces and wonderful models for toric arrangements}
\author{Aleksei Ilin and Leonid Rybnikov}

\begin{abstract}
    We study the family of commutative subspaces  in the trigonometric holonomy Lie algebra $t^{\mathrm{trig}}_{\Phi}$, introduced by Toledano Laredo, for an arbitrary root system $\Phi$. We call these subspaces \emph{Bethe subspaces} because they can be regarded as quadratic components of \emph{Bethe subalgebras} in the Yangian corresponding to the root system $\Phi$, that are responsible for integrals of the generalized XXX Heisenberg spin chain. Bethe subspaces are naturally parametrized by the complement of the corresponding toric arrangement . We prove that this family extends regularly to the minimal wonderful model $X_{\Phi}$ of the toric arrangement described by De Concini and Gaiffi \cite{cg}, thus giving a compactification of the parameter space for Bethe subspaces. For classical types $A_n, B_n, C_n, D_n$, we show that this extension is faithful. As a special case, when $\Phi$ is of type $A_n$, our construction agrees with the main result of Aguirre--Felder--Veselov on the closure of the set of quadratic Gaudin subalgebras. Our work is also closely related to, and refines in this root system setting, a parallel compactification result of Peters \cite{jp} obtained for more general toric arrangements arising from quantum multiplication on hypertoric varieties. Next, we show that the Bethe subspaces assemble into a vector bundle over $X_{\Phi}$, which we identify with the logarithmic tangent bundle of $X_{\Phi}$. Finally, we formulate conjectures extending these results to the setting of Bethe subalgebras in Yangians and to the quantum cohomology rings of Springer resolutions. We plan to address this in our next papers.

\end{abstract}

\maketitle

\section{Introduction}

\subsection{Holonomy Lie algebras} A holonomy Lie algebra is a Lie algebra defined by generators and relations which ensure the flatness of a certain  KZ type connection. At the same time, defining relations of a holonomy Lie algebra ensure commutativity of a certain set of Hamiltonians, i.e. the set of coefficients of a connection. 
There are different versions of holonomy Lie algebras, 
see for example \cite{afv}, \cite{at}, \cite{ikr}. 

In the present paper, we work with the \emph{trigonometric} holonomy Lie algebra $\mathfrak{t}_{\Phi}^{trig}$ introduced by Toledano Laredo in \cite{tl} corresponding to an arbitrary finite root system $\Phi$ of rank $n$ with a root lattice $R$.
The trigonometric holonomy Lie algebra is generated by the elements $t_{\alpha}, \alpha \in \Phi^+$  and $\tau(h), h \in \fh  = \Hom_{\bz}(R,\bc) \simeq \bc^n$ with relations 
$$[t_{\alpha}, \sum_{\beta \in \Psi} t_{\beta}] = 0 \quad \text{for any closed rank 2 root subsystem $\Psi\subset \Phi$};$$ 
$$ \tau(c_1h_1+ c_2h_2) = c_1\tau(h_1)+c_2\tau(h_2) \quad \text{for any $h_1, h_2 \in \fh,\ c_1,c_2\in\mathbb{C}$}; $$
 $$[\tau(h_1), \tau(h_2)] = 0 \quad \text{for any $h_1, h_2 \in \fh$};$$
$$[t_{\alpha}, \delta(h)] = 0 \quad \text{for any $h \in \fh$ such that $\alpha(h) = 0$}.$$ Here $\delta(h) = \tau(h) - \dfrac{1}{2} \sum_{\alpha \in \Phi^+} \alpha(h) t_{\alpha}$.
This relations ensure the flatness of the trigonometric Casimir connection on the torus $T = \Hom_{\bz}(R, \bc^{\times})$ with regular singularities at the arrangement of subtori determined by $\Phi$ (see \cite{tl}).

$$ \nabla_{trig} =  d -\sum_{\alpha \in \Phi^+} \frac{d \alpha}{e^{\alpha} - 1} t_{\alpha} - d h^i \tau(h_i) $$
Here, for any element $\alpha\in R$, $e^{\alpha}$ is the monomial function on $T$ corresponding to $\alpha$, and for any $\chi\in \fh^*$, $d\chi=\dfrac{de^\chi}{e^\chi}$ is the corresponding left-invariant $1$-form on $T$, $h^i$ is a basis of $\fh^*$, $h_i$ is the dual basis of $\fh$.

\subsection{Bethe subspaces}
There is a remarkable subspace $Q(C)$ of commuting elements of $\ft^{trig}_\Phi$ spanned by elements

$$BH(C,h) = \tau(h) - \sum_{\alpha \in \Phi^+} \dfrac{e^{\alpha}(C)\alpha(h)}{e^{\alpha}(C) - 1}  t_{\alpha}$$ 
where $h \in \fh$ and $C \in T^{reg} = T \setminus \cup_{\alpha \in T} (e^{\alpha} = 1)$ is a fixed regular element of $T$. We call these elements {\em Bethe Hamiltonians}.
Let $BH_i(C) = BH(C,h_i)$ be the basis elements of $Q(C)$.
Then one can rewrite trigonometric Casimir connection  in the form (see \cite[Proposition 2.2]{tl}):
\begin{equation*}\label{eq:trig-valerio}
     \nabla_{trig} = d - \tau(h^i)d h^i + \sum_{\alpha \in \Phi^+} \frac{d \alpha}{1 - e^{- \alpha}} t_{\alpha} = d - \sum_{i=1}^n BH_i(C) dh^i.
\end{equation*}

We call the subspace $Q(C)$ spanned by elements $B(C,h), h \in \fh$ a {\itshape Bethe subspace}. 
We have $\dim Q(C)=n$ for any $C\in T^{reg}$, so we obtain the map
$$\psi: T^{reg} \to Gr(n, \dim (t_{\Phi}^{trig})^{1})$$
taking $C\in T^{reg}$ to $Q(C)$

The main problem we address in the present paper is to to describe the closure of $\psi(T^{reg})$, i.e. we describe all possible limits of Bethe Hamiltonians and the corresponding parameter space. 

\subsection{De-Concini Gaiffi wonderful model for toric arrangements}

Let $\overline{T}$ be the smooth projective toric variety corresponding to the fan of Weyl chambers of $\Phi$. Consider the collection of subtori $T_{\alpha} = \{ C \in T \, | \, e^\alpha(C) = 1 \} $ and their closures $\overline{T_{\alpha}}$ inside $\overline{T}$. 

We consider the following {\em building set}: the set of locally indecomposable connected components of all possible intersections of  $\overline{T_{\alpha}}$ (see subsection \ref{pwm}). This is the minimal possible building set in the sense of De Concini and Gaiffi \cite{cg}.  Blowing up the toric variety $\overline{T}$ at the elements of the building set in the dimension increasing order, we obtain the (minimal) projective wonderful model $X_{\Phi}$. 
The  fiber of $X_\Phi$ over the unity point $1\in T\subset\overline{T}$ is the De Concini-Procesi wondeful model $M_\Phi$ for the hyperplane arrangement in the projective space $\mathbb{P}(\fh)$ determined by the root system $\Phi$, cf. \cite{DCP}.

\subsection{Extension of Bethe subspaces to $X_\Phi$} 
The first main result of the present paper is
\subsection*{Theorem A}
\emph{ \begin{enumerate}
     \item The family of trigonometric Bethe Hamiltonians extends to $X_{\Phi}$, i.e. the above morphism $\psi$ extends to a regular morphism $\overline{\psi}: X_\Phi\to Gr(n, \dim (t_{\Phi}^{trig})^{1}).$
     \item The morphism $\overline{\psi}$ onto its image is finite and generically bijective.
     \item For classical root systems $A_n,B_n,C_n$ and $D_n$, the morphism $\overline{\psi}$ is bijective onto its image.
  \end{enumerate}}

This means that the compactified parameter space for commutative Bethe subspaces is $X_\Phi$ up to nomalization. 

\begin{rem}
The first statement of Theorem~A is similar to the main result of Jeremy Peters \cite{jp} for arbitrary toric arrangement. However, in \cite{jp} the extension of the morphism $\psi$ is made to the closure corresponding to the particular building set described by Moci in \cite{m}. Our statement is then a refined version of Peters theorem in the case of an arrangement given by a root system, namely, we show that the map can be extended to the \emph{minimal} compactification of root toric arrangement corresponding to the minimal building set.
\end{rem}

\subsection{Trigonometric holonomy Lie algebra for $\Phi = A_n$}
We show that $\ft_{A_n}^{trig}$ is closely related to the algebra $\ft_{A_{n+1}}$ (modulo its center)  -- the rational holonomy Lie algebra, see section \ref{rathol} for definition.
Moreover,  Bethe Hamiltonians corresponds to Gaudin Hamiltonians considered in \cite{afv}, see Proposition \ref{a_n_case}.
According to \cite{afv}, the compactification of the family of Gaudin Hamiltoians is parameterized by $\overline{M}_{0,n+2}$ -- Deligne-Mumford compactification of the moduli space of stable rational curves with $n+2$ marked points.
This agrees with the well known fact that $\overline{M}_{0,n+2}$ is isomorphic to $X_{A_n}$, going back to Kapranov \cite{K} and Losev-Manin \cite{LM}, see e.g. \cite[Proposition~3.1]{imr}.

\subsection{Bethe bundle over $X_\Phi$}
In Section~6 we globalize the family of Bethe subspaces $Q(C)\subset (t^{\mathrm{trig}}_{\Phi})^{(1)}$ and show that they assemble into a vector bundle over the wonderful compactification $X_{\Phi}$. Namely, the map 
$\overline{\psi} : X_{\Phi} \to \gr(n,(t^{\mathrm{trig}}_{\Phi})^{(1)})$ 
determines the \emph{Bethe bundle} 
$\mathcal{Q} := \overline{\psi}^{*}(\mathcal{T})$, where $\mathcal{T}$ is the tautological rank $n$ bundle. Following the ideas of \cite{afv}, we prove

\subsection*{Theorem B}
\emph{The Bethe bundle $Q$ is naturally isomorphic to the logarithmic tangent bundle
\[
    \mathcal{Q} \;\cong\; T X_{\Phi}(-\log D),
\]
where $D = X_{\Phi} \setminus T^{\mathrm{reg}}$ is the boundary divisor.}

The immediate corollaries describe the restriction of $\mathcal{Q}$ to the central fiber $M_{\Phi}$ and identify it with the bundle of Gaudin subalgebras studied in~\cite{afv2}.  
More precisely, Theorem~B implies that the restriction of $Q$ to $M_{\Phi}$ is the sheaf $\mathcal{H}$ of Gaudin subalgebras and yields an explicit exact sequence
\[
0 \longrightarrow \mathcal{O}_{M_{\Phi}} \longrightarrow \mathcal{H} \longrightarrow 
T M_{\Phi}(-\log D) \longrightarrow 0,
\]
where $1 \mapsto c_{\Phi} = \sum_{\alpha\in\Phi^{+}} t_{\alpha}$ thus generalizing \cite[Theorem 3.3]{afv} from type A to arbitrary finite type root system.  

Any representation of $t^{\mathrm{trig}}_{\Phi}$ induces, via the symmetric algebra of $\mathcal{Q}$, a quotient sheaf of commutative algebras on $X_{\Phi}$ that is a quotient of the polynomials on its logarithmic cotangent bundle. Following \cite{afv}, we show that the kernel ideal of this quotient is stable with respect to the Poisson bracket. Thus the zero set of that ideal in the logarithmic cotangent bundle (i.e. the support of $\mathcal{Q}$ in the representation) is always coisotropic, and hence Lagrangian for finite-dimensional representations.

\subsection{Some remarkable representations of $\mathfrak{t}_\Phi^{trig}$} The trigonometric holonomy Lie algebra has several interesting homomorphisms to associative algebras. The image of Bethe subspaces under each of them generates a commutative subalgebra. We list here $4$ examples coming from Quantum Integrable Systems and Enumerative Geometry.

\subsubsection{Commutative subalgebras in Yangians} Let $\fg$ be the semisimple Lie algebra corresponding to the root system $\Phi$.
The Yangian $Y(\fg)$ for a complex simple Lie algebra is well-known deformation (as a Hopf algebra) of $U(\fg[t])$, the universal enveloping algebra of the current Lie algebra of $\fg$. In one of its presentations by generators and relations it is generated by elements $x, J(x)$ where $x \in \fg$ and $\fg,J(
\fg)$ generates two copies of adjoint representation inside $Y(\fg)$ with certain relations deformed those for $\fg, t \cdot \fg \subset U(\fg[t])$, see \cite{drin,wend}.
There is a homomorphism of Lie algebras 
$$\psi: t_{\Phi}^{trig} \to Y(\fg)$$ such that 
$$t_{\alpha} \mapsto  x_{\alpha}^+ x_{\alpha}^- + x_{\alpha}^- x_{\alpha}^+,  \alpha \in \Phi^+$$ 
$$\delta(h) \mapsto -2J(h).$$
Here $x_{\alpha}^\pm \in \fg$ are the Chevalley generators of $\fg$.
Moreover, $\psi$ is injective on the space of generators. The image of  $ Q(C) 
\subset t_{\Phi}^{trig}$ under $\psi$ is the quadratic part of a \emph{Bethe subalgebra} in the Yangian $Y(\fg)$, see \cite{ilin,ir2}. 
It shows that the question about compactification of Bethe Hamiltonians is closely related with a question about description of all possible limits of Bethe subalgebras in $Y(\fg)$, e.g. Bethe Hamiltonians is so-called quadratic part of a Bethe subalgebra in Yangian.

This example is a main motivation for us to study the trigonometric Gaudin Hamiltonians. In the forthcoming paper we will prove that the parameter space for Bethe subalgebras in $Y(\fg)$ coincides with $X_{\Phi}$ and will study the spectra of Bethe subalgebras in irreducible representations. This is closely related to the question about completeness of Bethe ansatz for the (generalized) XXX Heisenberg spin chain.

In type A, it is known from \cite{ir} that the compactification of the set of Bethe subalgebras in $Y(\fsl_n)$ is isomorphic to Deligne-Mumford compactification $\overline{M}_{0,n+2}$ and this agrees with results of the present paper since $\overline{M}_{0,n+2}=X_{A_n}$. 

In simply-laced types, the Yangian $Y(\fg)$ acts on equivariant cohomology of the corresponding quiver varieties with the operators from $Q(C)$ acting by quantum multiplication by $2$-dimensional classes \cite{mo}.

\subsubsection{Trigonometric Gaudin algebras} 
Let $\fg$ be any complex semisimple Lie algebra, $\fg = \fn_+ \oplus \fh \oplus \fn_-$ be its Cartan decomposition, $\{x_k \}_{k=1}^{\rk \fg}$ be an orthonormal basis with respect to the Killing form of $\fg$, $\Omega = \sum_k x_k 
\cdot x_k \in U(\fg)$ be the Casimir element, $$\Omega^{(ij)} = \sum_{k} 1 \otimes \ldots \otimes x_k \otimes \ldots  \otimes x_k \otimes \ldots \otimes 1  \in U(\fg)^{\otimes n}$$ where we put $x_k$ to $i$-th and $j$-th component.  Write $ \Omega = \Omega_+ + \Omega_0 + \Omega_-$ where $ \Omega_+ \in \fn_+ \otimes \fn_-$, $\Omega_0 \in \fh \otimes \fh$ and $ \Omega_- \in \fn_- \otimes \fn_+$.  
For $ \Phi = A_n$ the trigonometric Lie algebra $\ft^{trig}_{A_n}$ has homomorphisms to $U(\fg)^{\otimes n}$ for any complex simple Lie algebra.  It given by the following formula:

$$t_{ij} \mapsto \Omega^{(ij)}$$
$$\tau(\omega_k) \mapsto -\sum_{l=1}^k\sum_{c=1}^{l-1} \Omega^{(cl)} + \sum_{i=1}^k \sum_{j=1}^n \Omega^{(ij)}_-$$

Moreover, under this map Bethe Hamiltonians maps to (the linear combinations of) {\em trigonometric Gaudin Hamiltonians}, see sections \ref{rootA} and \ref{trig}.

\subsubsection{Commutative subalgebras in graded (degenerate) affine Hecke algebras} In  \cite{bmo10}, 
Braverman, Maulik and Okounkov describe how quantum multiplication by divisor classes 
on the Springer resolution $T^*G^\vee/B^\vee$ for the Langlands dual Lie group, can be represented inside the \emph{graded (degenerate) affine Hecke algebra} 
$H_\Phi$ associated with the root system $\Phi$. The second cohomology space of $T^*G^\vee/B^\vee$ is naturally identified with $\fh$. Let $D_h$ be the divisor class corresponding to $h\in\fh$. Then the operator of quantum multiplication by $D_h$ reads
\[
Q_h(q) = x_h + t \sum_{\alpha \in \Phi^+} 
\alpha(h)\, \frac{q^{\alpha}}{1 - q^{\alpha}}(s_\alpha - 1),
\]
where $x_h$ is the operator of (classical) multiplication by $D_h$, $s_\alpha$ are Weyl reflections, 
and $q^\alpha$ are K\"ahler parameters.  
The subspace spanned by these $Q_h(q)$ is commutative, 
describing the full family of quantum multiplication by divisors.

This family of commuting operators also comes from Bethe subalgebras in $\ft_\Phi^{trig}$ (with $C$ such that $q^\alpha=e^\alpha(C)$) via the homomorphism $\ft_\Phi^{trig}\to H_\Phi$ sending $t_\alpha$ to $s_\alpha-1$ and $\tau(h)$ to $-t^{-1}x_h$.

\subsubsection{Quantum cohomology of hypertoric varieties.}
In \cite{jp} Peters defines a representation of $\ft^{trig}_{\Phi}$ in the quantum cohomolgies of $X_{\Phi}$ such that $\tau(h)$ corresponds to the operators of classical multiplication by 2-dimesnional cohomology classes and $BH(C,h)$ corresponds to operators of quantum multiplication by 2-dimesnional cohomology classes.

\subsection{Some further questions and conjectures} The above examples suggest that the wonderful compactification $X_\Phi$ is not merely a convenient closure for trigonometric Bethe subspaces in the holonomy Lie algebra $t^{\mathrm{trig}}_\Phi$, but rather the intrinsic parameter space governing degenerations of Bethe-type commutative subalgebras across a broad range of representation-theoretic settings, thus leading to conjectures on their spectra regarded as coverings of the De Concini - Gaiffi compactification. Here we consider just one of such circles of conjectures, related to Bethe subalgebras in Yangians. 

\subsubsection{Extension to Bethe subalgebras in the Yangian}  Motivated by the embedding of $t^{\mathrm{trig}}_\Phi$ into the Yangian $Y(\mathfrak g)$ recalled in \S~1.7.1, we expect that the natural compactification of the parameter space of Bethe subalgebras $B(C)\subset Y(\fg)$ defined in \cite{ir2}, for $C\in T^{reg}$ is isomorphic to $X_\Phi$. More precisely, we conjecture that distinct points of $X_\Phi$ correspond to distinct limit commutative subalgebras in the Yangian $Y(\fg)$.

\subsubsection{Degeneration to inhomogeneous Gaudin and an additive limit of $X_\Phi$.}
It is well-known that $Y(\fg)$ is a flat deformation of $U(\fg[t])$, i.e. there is a family of algberas $Y_h(\fg) $ such that $Y_h(\fg) \simeq Y(
\fg), h \ne 0$ and $Y_0(\fg) = U(\fg[t])$. This allows us to degenerate commutative subalgebras in $Y(\fg)$ to those in $U(\fg[t])$.
It is shown in \cite{kmr} that Bethe subalgebras in Yangian degenerate to \emph{universal inhomogeneous Gaudin subalgebras} $\mathcal A^u_\chi\subset U(\fg[t])$ obtained from the Feigin--Frenkel center of $U(\widehat{\fg})$ at the critical level via quantum Hamiltonian reduction. These subalgebras depend on the parameter $\chi\in\fh$ and are maximal commutative for regular $\chi$. In \cite{kmr}, it is proved that  
\[
\lim_{h\to 0} B_h\bigl(\exp(h\chi)\bigr)=\mathcal A^u_\chi.
\]

We expect that this degeneration can be extended to the level of compactification. Namely, we expect certain \emph{additive degeneration} of the wonderful compactification $X_\Phi$ arising as the compactified parameter space of $A^u_\chi$ arising in \cite{kmr}, related to the additive degenerations of toric and wonderful models studied by Balibanu, Crowley and Li. For type $A$ this is done in \cite{ikr}.

\subsubsection{Compact real form of $X_\Phi$ and its fundamental group.} Let
\[
T^{comp}=\Hom_{\mathbb Z}(R,S^1)\subset T
\]
be its compact real form. Denote by $T^{comp}_{reg}$ the complement in $T^{comp}$ of the real
root hypertori $\{e^\alpha=1\}$, $\alpha\in\Phi$. We define the \emph{compact real form} of the
wonderful compactification by
\[
X_\Phi^{comp}:=\overline{\,T^{comp}_{reg}\,}\subset X_\Phi,
\]
where the closure is taken in the complex analytic topology.
By construction, $X_\Phi^{comp}$ is a real semialgebraic subset of $X_\Phi$, stable under the
action of the Weyl group $W$, and its stratification is induced from the boundary stratification
of $X_\Phi$ by root subsystems.

The $W$--action on $X_\Phi^{comp}$ gives rise to the $W$--equivariant fundamental group \(
\pi_1^{W}\!\bigl(X_\Phi^{comp}\bigr),
\) which fits into a natural extension
\[
1 \longrightarrow \pi_1\!\bigl(X_\Phi^{comp}\bigr)
\longrightarrow \pi_1^{W}\!\bigl(X_\Phi^{comp}\bigr)
\longrightarrow W
\longrightarrow 1.
\]
Generalizing \cite{iklpr}, we expect that this group contains the \emph{affine $W$-cactus group} $AC_W$ as a subgroup of finite index. More precisely, $AC_W$ is associated with the affine root system $\Phi_{aff}$ and is generated by involutions
\[
\tau_I,\qquad I\subset \Delta_{aff},
\]
indexed by connected subsets of $\Delta_{aff}$, the (affine) Dynkin diagram of $\Phi_{aff}$, subject to the relations:
\begin{enumerate}
\item $\tau_I^2=1$ for all $I$;
\item $\tau_I\tau_J=\tau_J\tau_I$ whenever the corresponding affine root subsystems are orthogonal;
\item $\tau_I\tau_J=\tau_J\tau_{w_J(I)}$ whenever $I\subset J$, where $w_J$ is the longest element
of the affine Weyl group of the subsystem generated by $J$.
\end{enumerate}
The natural projection $AC_W\to W$ sends $\tau_I$ to $w_I$, and its kernel
\[
PAC_W := \pi_1\!\bigl(X_\Phi^{comp}\bigr)
\]
is the \emph{pure affine cactus group}. This presentation generalizes the finite
cactus group introduced by Losev to the affine setting, and is expected to coincide with the
affine cactus group $AC_n$ from \cite{iklpr} in type $A_{n-1}$.

\subsubsection{Crystal structure and monodromy conjecture.} Let $V=\bigotimes_i W_{k_i,r_i}(u_i)$ be a tensor product of Kirillov--Reshetikhin $Y(\fg)$-modules satisfying suitable genericity and multiplicity--free assumptions. Results of \cite{kmr} show that, in such cases, Bethe subalgebras $B(C)\subset Y(\fg)$ act cyclically on $V$, and that under the degeneration $C=\exp(\varepsilon\chi)$, $\varepsilon\to 0$, their spectra degenerate to those of the universal inhomogeneous Gaudin subalgebras $\mathcal A^u_\chi\subset U(\fg[t])$. Following \cite{kmr}, we expect that this cyclicity holds over the whole compactified parameter space $X_\Phi$, and thus the joint spectrum of $B(C)$ on $V$ forms a (possibly ramified) covering of  $X_\Phi$. Further, we expect that this covering is unramified over the real locus $X_\Phi^{comp}$ and the fiber of this covering over a generic point can be naturally identified with the affine type Kirillov-Reshetikhin crystal of $V$, generalizing the results of \cite{kmr}. In particular, in the additive degeneration limit $C=\exp(\varepsilon\chi)$, $\varepsilon\to 0$, this crystal structure is expected to specialize to the crystal structure on the spectrum of the universal inhomogeneous Gaudin subalgebra $\mathcal A^u_\chi$ on a tensor product of evaluation representations, see \cite{hkrw,kr}. 

The above covering of $X_\Phi^{comp}$ is equivariant with respect to the Weyl group $W$, so there is a natural action of the equivariant fundamental group $\pi_1^W(X_\Phi^{comp})$ on the fiber. We expect that the restriction of this group action to $AC_W$ is the \emph{inner cactus group action} on the Kirillov-Reshetikhin crystal being the straightforward generalization of the inner cactus group on finite-type Kashiwara crystals, see \cite{hkrw,hlly}.  

\subsection{The paper is organized as follows}
In section~2 we recall definitions of rational and trigonometric holonomy Lie algebras. In section~3 we discuss compactification of the family of Gaudin Hamiltonians and the corresponding parameter space: De Concini - Processi compactification $M_{\Phi}$. In section~4 we recall definition of De Concini - Gaiffi wonderful projective model for a toric arrangement and define its open cover. In section~5 we prove Theorem~A. In section~6 we prove Theorem B and its consequences.

\subsection{Acknowledgements} We thank Giovanni Gaiffi, Joel Kamnitzer and Jeremy Peters for useful discussions. The research of A.I. is supported by MSHE RF GZ project. The work of L.R. was supported by the Fondation Courtois.

\section{Holonomy Lie algebras and Bethe subspaces}
\label{lie_alg}

\subsection{Notation} Throughout the paper $\Phi$ is a reduced, finite root system of the rank $n$, $\Phi^+ \subset \Phi$ is a set of positive roots, $\Delta \subset \Phi^+$ is a set of simple roots, $R=\bz\Delta$ is the root lattice, $\fg$ is the complex semisimple Lie algebra corresponding to $\Phi$, such that $\fh=\Hom_\bz(R,\bc)$ is its Cartan subalgebra. 
We say that a root subsystem $\Phi^{\prime} \subset \Phi$ is a closed root subsystem if the following is true: $\alpha,
\beta \in \Phi^{\prime}$ and $\alpha+\beta \in \Phi$ then $\alpha+\beta \in \Phi^{\prime}$.

\subsection{Holonomy Lie algebra $\mathfrak{t}_\Phi$ and Gaudin Hamiltonians}
\label{rathol}

\begin{defn} (cf. \cite{afv2})
The holonomy Lie algebra $t_{\Phi}$ is the Lie algebra generated by the elements
$t_{\alpha}, \alpha \in \Phi$ such that $t_{-\alpha} = t_{\alpha}$ with the following relations.
For any rank $2$ subsystem $\Psi$ such that $\Psi = V \cap \Phi$ for some $2$-dimensional subspace $V$ and $\alpha \in \Psi$ we have
$$[t_{\alpha}, \sum_{\beta \in \Psi} t_{\beta}] = 0$$

\end{defn}
\begin{rem}
The element $c_\Phi=\sum_{\alpha \in \Phi_+} t_{\alpha}$ belnons to the center of $\ft_{\Phi}$.
\end{rem}
For any $\alpha \in \Phi_+$ denote by $D_{\alpha} \subset \fh$ hyperplane given by equation $\alpha(h) = 0$. Consider a fixed element $\chi \in \fh^{reg} = \fh \setminus \cup_{\alpha \in \Phi_+} D_{\alpha}$. 
Denote by $t_{\Phi}^1$ the subspace of $t_{\Phi}$ generated by generators and consider the following elements of $t_{\Phi}^{1}$ corresponding to any $h \in \fh$:

$$H(h, \chi) = \sum_{\alpha \in \Phi_+} \dfrac{\alpha(h)}{\alpha(\chi)} t_{\alpha}$$
The elements above are called Gaudin Hamiltonians.
From  relations in $t_{\Phi}$, it follows that $[H(h_1), H(h_2)] =0$ for any $h_1, h_2 \in \fh$, see e.g. \cite{k}. Denote by $G(\chi)$ the subspace spanned by all $H(\chi, h), h \in \fh$. 
\begin{lem} 
\begin{itemize}
\item For any $\chi\in\fh^{reg}$, we have $\dim G(\chi)=n$;
\item If $\Phi$ is an irreducible root system then for different $\chi \in \bp(\fh^{reg})$ the subspaces $G(\chi)$ are different.
\end{itemize}
\end{lem}
\begin{proof}
Let $\alpha_i$ be a set of simple roots and $h_i \in \fh$ such that $\alpha_i(h_j) = \delta_{ij}$. Then
the projection of the linear span of $H(h_i, \chi), i=1,\ldots,\rk\Phi$ to the linear span of $\left<t_{\alpha_i}\right> \subset \ft_{\Phi}$ has the dimension $\rk \fg$ because the projection of $H(h_i,\chi)$ is just $t_{\alpha_i}$. Then $\dim G(\chi) = \rk \fg$.

For the second statement consider the preimage of $t_{\alpha_i}$ under projection.  It has form $$t_{\alpha_i} + \sum_{\beta \in \Phi, (\beta,\omega_i)\ne0}  \dfrac{\beta(h_i) \alpha_i(\chi)}{\beta(\chi)}.$$ 
This determines $\chi$ uniquely up proportionality since the root system $\Phi$ is irreducible.
\end{proof}

Consider the map 
$$\iota: \bp(\fh^{reg}) \to Gr(n, \dim t_{\Phi}^1)$$
taking $\spann(\chi)\in \mathbb{P}(\fh)$ to  $G(\chi)\in Gr(n, \dim t_{\Phi}^1)$. The following is the one of the main results of \cite{afv2}.

\begin{thm}
\label{afv12}
The closure of $\iota(\bp(\fh^{reg}))$ in the above Grassmannian is isomorphic to the De Concini-Procesi compactification $M_{\Phi}$ (see Section \ref{dcp-s}).
\end{thm}


\subsection{Trigonometric holonomy Lie algebra $\mathfrak{t}_\Phi^{trig}$ and Bethe Hamiltonians}
In the next definition we follow \cite{tl}.
\begin{defn}
The trigonometric holonomy Lie algebra $t_{\Phi}^{trig}$ is the Lie algebra generated by the elements
$t_{\alpha}, \alpha \in \Phi$ and $\tau(h), h \in \fh$ such that $t_{-\alpha} = t_{\alpha}$ with the following relations

1) For any rank $2$ closed root  subsystem $\Psi \subset \Phi$ and $\alpha \in \Psi$ we have
$$[t_{\alpha}, \sum_{\beta \in \Psi} t_{\beta}] = 0;$$

2) $\tau(c_1 h_1+ c_2 h_2) = c_1 \tau(h_1) + c_2 \tau(h_2)$ for any $h_1, h_2 \in \fh$, $c_1,c_2 \in \bc$; \\

3) $[\tau(h_1), \tau(h_2)] = 0 $ for any $h_1, h_2 \in \fh$;
 
4) $[t_{\alpha}, \delta(h)] = 0$ for any $h \in \fh$ such that $\alpha(h) = 0$. Here $\delta(h) = \tau(h) - \dfrac{1}{2} \sum_{\alpha \in \Phi^+} \alpha(h) t_{\alpha}$.

\end{defn}

From definition it follows that $\delta(c_1h_1+ c_2h_2) = c_1\delta(h_1) + c_2\delta(h_2)$ for any $h_1, h_2 \in \fh$ and $c_1,c_2\in \mathbb{C}$.

\begin{prop}\cite{tl}
\label{tl2}
    The algebra $\mathfrak{t}_\Phi^{trig}$ carries the following $W$-action: 
    $$w \cdot t_{\alpha} = t_{w \cdot \alpha}$$
    $$w \cdot \tau(h) = \tau(h) - \sum_{w \in \Phi_+ \cap w \Phi_-} \alpha(h) t_{\alpha}$$
    In particular, we have $w \cdot \delta(h) = \delta(h)$ for any $w \in W$.
\end{prop}

We also note that we there is the obvious homomorphism from $\ft_{\Phi}$ to $\ft^{trig}_{\Phi}$ as well as a homomorphism from $\ft_{\Phi^{\prime}}$ to $\ft^{trig}_{\Phi}$ where $\Phi^{\prime} \subset \Phi$ is a closed root subsystem.


Let $C \in T^{reg}$ and consider the following elements of $t_{\Phi}^{trig}$:
$$BH(C,h) = \delta(h) - \dfrac{1}{2}\sum_{\alpha \in \Phi^+} \dfrac{e^{\alpha}(C)+1}{e^{\alpha}(C) - 1} \alpha(h) t_{\alpha} = \tau(h) - \sum_{\alpha \in \Phi^+} \dfrac{e^{\alpha}(C)\alpha(h)}{e^{\alpha}(C) - 1}  t_{\alpha}$$
From  relations in $t_{\Phi}^{trig}$ it follows, that $[BH(h_1), BH(h_2)] =0$ for any $h_1, h_2 \in \fh$, see \cite[Section 3.8]{tl}.
Denote by $Q(C)$ the linear span of all the $BH_i(C)$.

\begin{lem}

\begin{itemize}
\item $\dim Q(C)=\rk \Phi$;
\item For different $C \in T^{reg}$ the subspaces $Q(C)$ are different.
\end{itemize}
\end{lem}
\begin{proof}
First statement follows from the fact that dimension of the projection of the span $BH_i(C)$   to span of $\tau(h_i)$ is equal to $\rk \Phi$.
The second follows from the fact that the subspace $Q(C)$ is uniquely determined by its projection to $\spann \langle \tau(h) \rangle$.
\end{proof}

\subsection{The case of the root system of type $A_n$}
\label{rootA}

Recall that in type $A_n$ a set of positive roots can be chosen as $\{ \varepsilon_i - \varepsilon_j \, | \, i \ne j, i,j = 1, \ldots, n+1\}$. We denote by $t_{ij}$ the element $t_{ \varepsilon_i - \varepsilon_j } \in \ft_{A_n}$. 

\begin{prop}
\label{a_n_case}
\begin{enumerate}
    \item There is a homomorphism $\psi$ from trigonometric holonomy Lie algebra $\mathfrak{t}_{A_n}^{trig}$ to $\mathfrak{t}_{A_{n+1}}$ which is an injection on the space of generators.
    \item Let $\pi: \mathfrak{t}_{A_{n+1}} \to \mathfrak{t}_{A_{n+1}}/ \left<\sum_{i,j=0}^n t_{ij} \right>$ be the projection onto the quotient by the central element $\sum_{i,j=0}^n t_{ij}$. Then $\pi \circ \psi: \mathfrak{t}_{A_n}^{trig} \to \mathfrak{t}_{A_{n+1}}/ \left<\sum_{i,j=0}^n t_{ij} \right>$ is a surjective homomorphism. 
    \item Let $C = \diag(z_1, \ldots, z_n)$. Then $\psi(Q(C)) = G(\chi)$ with $\chi=\diag(0,z_1,\ldots,z_n)$.
\end{enumerate}
\end{prop}
\begin{proof}
The map is given by the formula
$$\psi: \mathfrak{t}^{trig}_{A_n}  \to \mathfrak{t}_{n+1}$$

$$t_{ij} \mapsto t_{ij}, 1 \leq i,j \leq n$$
$$\tau(\omega_k) \mapsto -\sum_{l=1}^k\sum_{c=0}^{l-1} t_{cl}$$ 

Here $\omega_k$ are elements of $\fh$ corresponding to the fundamental weights.
Since the images of $\tau(\omega_k)$ are linear combinations of the Jucys-Murphy elements, they pairwise commute.
We have 
$$\delta(\omega_k) \mapsto  -\sum_{l=1}^k\sum_{c=0}^{l-1} t_{cl} - \frac{1}{2} \sum_{i \in \{ 1, \ldots, k\}, j \in \{ k+1, \ldots, n \}} t_{ij}$$

It commutes with $t_{ij}$ for $i,j \in \{1, \ldots, k \}$ or $i,j \in \{k+1, \ldots, n\}$. The map is clearly injection on the set of generators and composition with $\pi$ is clearly surjective.

To prove the last claim we observe that
$$BH(\omega_k,C) \mapsto  -\sum_{l=1}^k\sum_{c=0}^{l-1} t_{cl} - \sum_{i \in\{1, \ldots, k\}, j \in \{ k+1, \ldots, n \} } \frac{z_i}{z_i-z_j} t_{ij}$$

Then $BH(\omega_1,C) + \ldots + BH(\omega_k,C) \mapsto -z_k H_k(0,z_1, \ldots, z_n)$, where 
$$H_k(0,z_1, \ldots, z_n) = \frac{t_{0k}}{z_k} + \sum_{i = 1, i \ne k}^n \dfrac{t_{ki}}{z_k-z_i} $$
and the claim follows because $H_1, \ldots H_n$ span $G(0,C)$.

\end{proof}

\subsection{Gaudin Hamiltonians in $U(\fg)^{\otimes n}$.}
\label{trig}
The \emph{rational Gaudin Hamiltonians} \cite{g1,g2} are the following commuting elements of $U(\fg)^{\otimes n}$ depending on pairwise different complex numbers $z_1,\ldots,z_n$:
\[
H_i:=\sum\limits_{j\ne i}\dfrac{\Omega_{ij}}{z_i-z_j}\quad i=1,\ldots,n.
\]
They are the images of Gaudin Hamiltonians in $\ft_{n}=\ft_{A_{n-1}}$ under the homomorphism $\phi:\mathfrak{t}_{A_{n-1}} \to U(\fg)^{\otimes n}$ that takes $t_{ij} \mapsto \Omega^{(ij)}$ for $1\le i,j\le n$. The only linear dependence on them is still $\sum\limits_{i=1}^n H_i=0$, so their linear span is $(n-1)$-dimensional.

In \cite{ikr}, we consider the set of commuting \emph{trigonometric Gaudin Hamiltonians} in $U(\fg)^{\otimes n}$ attached to any $\theta\in\fh$:
$$ H_{i, \theta}^{trig} = \frac{\theta^{(i)}}{z_i} + \sum_{j\ne i } \frac{\Omega^{(ij)}}{z_i - z_j} - \sum_j \frac{\Omega_-^{(ij)}}{z_i}$$
Here $\theta^{(i)}$ denotes $\theta$ inserted to the $i$-th tensor factor.
More specifically, we have rational Gaudin Hamiltonians $H_1,\ldots,H_n$ in $U(\fg)^{\otimes (n+1)}$, the images of Gaudin Hamiltonians in $\ft_{n+1}=\ft_{A_n}$ under the homomorphism $\phi:\mathfrak{t}_{A_{n}} \to U(\fg)^{\otimes (n+1)}$ that takes $t_{ij} \mapsto \Omega^{(ij)}$ for $0\le i,j\le n$. Next, we have a homomorphism 
$i_{\theta}: (U(\fg)^{\otimes n+1})^{\fg} \to U(\fg)^{\otimes n}$ (see \cite[Section 5.1]{ikr} for the definition) such that $H_{i, \theta}^{trig}=i_{\theta}(H_i)$ for $i=1,\ldots,n$. 

According to \cite{ikr}, the element $\Omega_{0i}, i = 1, \ldots, n$ goes to $\theta^{(i)} - \sum_{j=1}^n \Omega_-^{(ij)}$ under the map $i_{\theta}$.
Composing it with the homomorphisms $\phi:\mathfrak{t}_{A_{n}} \to U(\fg)^{\otimes (n+1)}$ and $\psi: \mathfrak{t}_{A_n}^{trig} \to \mathfrak{t}_{A_{n+1}}$, we obtain the map
$$ i_{\theta} \circ \phi \circ \psi: \mathfrak{t}_{A_n}^{trig} \to U(\fg)^{\otimes n}$$
$$t_{ij} \mapsto \Omega^{(ij)}$$
$$\tau(\omega_k) \mapsto -\sum_{l=1}^k\sum_{c=1}^{l-1} \Omega^{(cl)} + \sum_{i=1}^k \sum_{j=1}^n \Omega^{(ij)}_- - \sum_{m=1}^k \theta^{(m)}$$
So we get the following 

\begin{prop} Let $C=(z_1,\ldots z_n)$ with $z_i$ being pairwise different complex numbers. Then the homomorphism $i_{\theta} \circ \phi \circ \psi$ takes $B(\omega_1,C) + \ldots + B(\omega_k,C)$ to $-z_k H_{k, \theta}^{trig}$.
\end{prop}

\subsection{Parameter space and compactification problem.}
The subspace $Q(C)$ generated by Bethe Hamiltonians define the map 
$$\psi: T^{reg} \to Gr(n, \dim (t_{\Phi}^{trig})^{1}).$$
The main purpose of present paper is to describe the closure of $T^{reg}$ in Grassmannian. One can think about this as multiplicative analog of the same problem for quadratic Gaudin Hamiltonians in the holonomy Lie algebra $\mathfrak{t}_{\Phi}$, see \cite{afv2}.

We consider $T^{reg}$ as the complement of an arrangement of subtori in $T$:  $$T^{reg} = T \setminus \bigcup_{\alpha \in \Phi^+}T_{\alpha}.$$ Following \cite{cg} using this data one can construct a projective wonderful compactification $X_{\Phi}$. In section 5 we will prove that the closure is isomorphic (up to normalization) to $X_{\Phi}$ for root systems of types $A,B,C,D$ and also discuss what happens for the exceptional root systems.



\section{The variety $M_\Phi$ and Gaudin Hamiltonians}
\label{dcp-s}
\subsection{Definition of $M_\Phi$} Let us recall the definition of the De Concini-Procesi wonderful model for the hyperplane arrangement in $\fh$ determined by an \emph{indecomposable} root system $\Phi$. 

Let $ \mathcal{G}' $ denote the set of all non-zero subspaces of $ \fh^* $ which are spanned by a subset of $ \Phi $.   Let $ V \in \mathcal{G}' $.  We say that $ V  = V_1 \oplus \cdots \oplus V_k $ is a \emph{decomposition} of $ V $ if  $ V_1, \dots, V_k \in \mathcal{G}'$, and if whenever $ \alpha  \in \Phi $ and $ \alpha \in V $, then $ \alpha \in V_i $ for some $ i$.  From Section 2.1 of \cite{DCP}, every element of $ \mathcal{G}' $ admits a unique decomposition. Define the set $ \mathcal{G} $ as the set of indecomposable elements of $ \mathcal{G}' $. So any element of $\mathcal{G}'$ can be uniquely decomposed into an orthogonal direct sum of elements from $\mathcal{G}$. We call such decomposition a $\mathcal{G}$-\emph{decomposition.}

\begin{defn}
The De Concini-Procesi space $M_\Phi \subset \prod_{V \in \mathcal{G}} \mathbb{P}(\fh/V^\perp) $ is the closure of the image of the map $$ \bp{\fh^{reg}} \rightarrow  \prod_{V \in \mathcal{G}} \mathbb{P}(\fh/V^\perp).$$ 
\end{defn}
According to \cite{DCP2} the variety $M_{\Phi}$ is smooth, and, moreover, it is obtained from $\bp{\fh}$ by subsequently blowing up the dual subspaces of elements from $\mathcal{G}$, ordered from smaller to bigger dimensions. Let $\widetilde{M}_\Phi$ be the non-compact version of $M_\Phi$, i.e. the closure of $\fh^{reg}$ in  $ \fh\times\prod_{V \in \mathcal{G}} \mathbb{P}(\fh/V^\perp)$. It is obtained from the affine space $\fh$ by the same process of blowing up the dual subspaces of elements from $\mathcal{G}$, ordered from smaller to bigger dimensions. The projection onto the first factor, $\widetilde{M}_\Phi\to\fh$, is, clearly, birational map, such that the fiber over $0\in\fh$ is $M_\Phi$. Let $\fh^{loc}$ be the formal neighborhood of $0\in\fh$. We will also need the local version of $\widetilde{M}_\Phi$, namely
\[
\widetilde{M}_\Phi^{loc}:=\widetilde{M}_\Phi\times_{\fh}\fh^{loc}
\]

\subsection{$M_\Phi$ for a decomposable $\Phi$} For $\fh=\bigoplus\limits_{i=1}^s\fh_i$ and $\Phi=\bigcup\limits_{i=1}^s\Phi_i$ with $\Phi_i=\Phi\cap\fh_i^*$, we define $$
M_\Phi:=\prod\limits_{i=1}^sM_{\Phi_i}.
$$
Note that this is compatible with the definition of $\widetilde{M}_\Phi$ as the subsequent blow-up of $\fh$ at all indecomposable intersections of the hyperplanes $D_\alpha$ for $\alpha\in\Phi$. Namely, the fiber of $\widetilde{M}_\Phi\to\fh$ at $0\in\fh$ is $M_\Phi$. Similarly, we define 
$\widetilde{M}_\Phi^{loc}:=\widetilde{M}_\Phi\times_{\fh}\fh^{loc}$ for any $\Phi$, not necessarily indecomposable.

\subsection{Charts on $M_\Phi$}
A subset $  S  \subset \mathcal{G} $ is called \emph{nested} if, whenever $ P_1, \dots, P_k \in  S  $ are pairwise incomparable (with respect to inclusion), then $ P_1+\ldots+P_k=P_1 \oplus \cdots \oplus P_k $ is a $\mathcal{G}$-decomposition. A subset is called \emph{maximal nested} if it is not contained in any other nested subset . It is known that every maximal nested set has  $\dim \fh$ elements and contains $ \fh^* $. Moreover, nested sets have the following property:
\begin{lem}
Let $\alpha \in \Phi$. Then
the subsets of $S$ containing $\alpha$ are linearly ordered and there is a unique minimal
 $A_S(\alpha) \in S$ such that $\alpha \in A_S(\alpha)$.
\end{lem}

It follows, that any nested set $S$ defines a partial order on $\Phi^+$:
$\alpha < \beta \iff A_S(\alpha) \subset A_S(\beta)$. Let $S$ be a nested set.

\begin{defn}\label{def:charts-on-mphi}
The chart $W_S$ is an open subset of $M_{\Phi}$ determined by the following property:
for every $P \in S$ consider $P_1, \ldots, P_k \in \mathcal{G}$  -- maximal (proper) subsets of $P$. Then for $z_P \in P(\fh/V^{\perp})$ we have
$\alpha(z_P) \ne 0$ for any $\alpha \in P \setminus \cup P_i$.
\end{defn}

\begin{prop}\label{pr:DCP-open-cover} \cite{DCP2}
The sets $W_S$, where $S$ is a nested set, form an open covering of $M_{\Phi}$.

\end{prop}

Following \cite[Section 4.1]{afv}, nested sets in a root arrangement are uniquely determined by the following data:
\begin{itemize}
\item Choice of a system of simple roots;
\item Choice of a nested set of subgraphs in the Coxeter graph corresponding to $\Phi$. That is a set of connected subgraphs such that any two of them are either disjoint (in the original graph) or contained one in another.
\end{itemize}
Given such data, the linear spans of simple roots corresponding to vertices of subgraphs from the nested set of subgraphs in the Coxeter graph, form a nested set in the corresponding root arrangement of subspaces. Moreover, any nested set has such form.  

Next, any of such nested sets comes with a  partial order on all roots, as pointed out above.
We will mostly use the following property of $W_S$:
\begin{prop}\label{pr:DCP-reg-functions} \cite[Section 1.4]{DCP}
The subset $W_S$ has the following property: for $\alpha,\beta\in \Phi$, the function $w_{\beta,\alpha}=\dfrac{\beta}{\alpha}$ is regular on $W_S$ if 
$\alpha>\beta$ with respect to the partial order.
\end{prop}

For what follows, we will need the open charts $\widetilde{W}_S^{loc}\subset\widetilde{M}_\Phi^{loc}$ whose definition is the same as  Definition~\ref{def:charts-on-mphi} with the only change $M_\Phi$ to $\widetilde{M}^{loc}_\Phi$. 
Then Propositions~\ref{pr:DCP-open-cover}~and~\ref{pr:DCP-reg-functions} clearly hold true for $\widetilde{W}_S^{loc}$ as well.

\subsection{Parameter space for Gaudin subspaces} 
\label{AFV_proof}
According to \cite{afv2}, the family of Gaudin subspaces in $\ft_\Phi^1$ extends to $M_\Phi$. For readers' convenience and in order to use it in the proof in the trigonometric case, we reproduce the proof here, in a manner slightly different to that of \cite{afv2}.

\begin{prop}\label{pr:AFV}\cite[Proposition 3.1]{afv2}
The embedding $\iota: \mathbb{P}(\fh^{reg})\to Gr(r, \ft_{\Phi}^{(1)})$ uniquely extends to a map  $\overline{i}: M_{\Phi}  \to Gr(r, \ft_{\Phi}^{(1)})$.
\end{prop}
\begin{proof}
Let $S$ be a maximal nested set, $W_S$ be the corresponding chart and $x \in W_S$. Denote by $\chi $ the corresponding element of $\fh^{reg}$ and by $\beta_i, i = 1, \ldots, \rk \fg$ -- adapted basis of $\fh^*$ corresponding to a nested set $S$ and by $h_i$ the corresponding dual basis 
of $\fh$. 
Consider the following set of Hamiltonians:
$$x \mapsto H_i(x) = \sum_{\alpha \in \Phi_+} \alpha(h_i) w_{\beta_i \alpha}(x)t_\alpha$$

Note that if $x = \chi \in \fh^{reg}$ then
$$H_i(x) = H(h_i,\chi) \cdot \beta_i(\chi),$$
hence the subspace spanned by $H_i(x)$ it this case coincides with original subspace spanned by $H(h_i, \chi)$.

We either have $\alpha(h_i) = 0$ or root $\alpha$ contains $\beta$ in its decomposition in a basis of simple roots. In the latter case $\alpha > \beta_i$ hence $H_i(x)$ is well-defined for $x\in W_S$. 
Moreover the projection of the linear span of $H_i(x), i=1,\ldots,\rk\fg$ to the linear span of $\left<t_{\beta_i}\right> \subset \ft_{\Phi}$ has the dimension $\rk \fg$ because the projection of $H_i(x)$ is just $t_{\beta_i}$. 
Hence this extends the embedding $\iota$ to $M_{\Phi}$. The uniqueness follows from the fact that $\bp(\fh^{reg}) \subset M_{\Phi}$ is Zariski dense.

\end{proof}

Moreover, the following is true:

\begin{prop}\cite[Theorem 4.4]{afv2}
The map  $\overline{\iota}: M_{\Phi} \to Gr(r, \ft_{\Phi}^{1})$ is a closed embedding.
\end{prop}

\section{The variety $X_{\Phi}$}

\subsection{Toric variety $\overline{T}$ and its affine charts} Let $\overline{T}$ be the smooth projective toric variety corresponding to the fan of Weyl chambers of $\Phi$. It has an open covering by affine spaces  $U_\Delta$ corresponding to systems of simple roots $\Delta\subset\Phi$. Consider the following smaller open subsets attached to any such $\Delta$ and any subset $I\subset \Delta$: 

$$ U_{\Delta,I}=\{x\in\overline{T}\ |\ y_i(x)\ne0 \ \forall i\in I\}.
$$

The subset $U_{\Delta,I}$ contains the smallest $T$-orbit $O_{\Delta,I}\subset \overline{T}$ determined by $y_j=0$ for all $j\not\in I$. The stationary subgroup of any point of $O_{\Delta,I}$ in $T$ is $Z_I:=\{t\in T\ |\ e^{\alpha_i}(t)=1\ \forall \alpha_i\in I\}$, i.e. the center of the Levi subgroup $L_I\subset G$ corresponding to $I$. So the orbit $O_{\Delta,I}$ gets naturally identified with the maximal torus $T_I:=T/Z_I$ in the adjoint group $L_I^{\ad}:=L_I/Z_I$. 

\begin{prop}
    The open set $U_{\Delta, I}$ is naturally $O_{\Delta,I}\times E_{\Delta,I}=T_I\times E_{\Delta,I}$ where the toric coordinates on $T_I$ are $y_i, i\in I$ and $E_{\Delta,I}$ is the vector space with the coordinates $y_j, j\not\in I$.
\end{prop}

We will make use of the following finer open covering of $\overline{T}$. Let $I\subset \Delta$ and $\Phi_I$ be the root subsystem in $\Phi$ generated by $I$. Define $ U_{\Delta,I}^{loc}$ as the formal neighborhood of $T_{I}\times\{0\}$ in $U_{\Delta, I}=T_I\times E_{\Delta,I}$. So it is the Cartesian product 
$$ U_{\Delta,I}^{loc}=T_I\times E^{loc}_{\Delta,I},
$$
where $E^{loc}_{\Delta,I}$ is the formal neighborhood of the origin in $E_{\Delta,I}$. Note that $O_{\Delta,I}$ is contained in $U_{\Delta,I}^{loc}$, and, on the other hand, any $T$-orbit in $\overline{T}$ is $O_{\Delta,I}$ for some $\Delta$ and $I$. So this is indeed an open covering. We have the following:

\begin{prop}\label{pr:restriction-to-Levi}
    The intersection of the arrangement $\{ \overline{T_\alpha}\ |\ \alpha\in\Phi\}$ in $\overline{T}$ with $U_{\Delta,I}^{loc}$ is $\{\overline{T'_\beta}\times E^{loc}_{\Delta,I}\ |\ \beta\in\Phi_I\}$, where $\{T'_\beta\ |\ \beta\in\Phi_I\}$ is the toric arrangement in $T_I$ corresponding to the root system $\Phi_I$ and $E_{\Delta,I}$ is the complement affine space.
\end{prop} 

\begin{proof} The subvarieties $\overline{T_\beta}$ with $\beta\not\in\Phi_I$ do not touch $O_I$, so we have
$$U_{\Delta,I}^{loc}\subset U_{\Delta,I}\setminus \bigcup\limits_{\beta\not\in\Phi_I}\overline{T_\beta}=\{x\in\overline{T}\ |\ y_i(x)\ne0 \ \forall i\in I\ \text{and} \ y_\alpha(x)\ne1 \ \text{for} \ \alpha\not\in\Phi_I\}.
$$ Moreover, for all $\beta\in\Phi_I$, we have $\overline{T_\beta}\cap U_{\Delta,I}^{loc}=\overline{T'_\beta}\times E_{\Delta,I}^{loc}$.
\end{proof}

\subsection{De Concini -- Gaiffi wonderful model} 
\label{pwm}
We recall definition of De Concini-Gaiffi wonderful model following \cite{cg}. Let $\cg^{\prime}$ be the set of {\em layers} i.e. the connected components of all the intersections of the closures
of $\overline{T_{\alpha}}$ inside $\overline{T}$. 
This set forms an arrangement of subvarieties inside $\overline{T}$.
We say that element of $Y \in \cg^{\prime}$ is decomposable if
$Y$ is locally a transversal intersection of $Y_1$ and $Y_2$ such that for any $\alpha \in \Phi^+$ such that $Y \subset \overline{T}_{\alpha}$ then $Y_1\subset \overline{T}_{\alpha} $ or $Y_2\subset \overline{T}_{\alpha} $.  That is, to any layer $Y$ we assign an open subset $U_Y^{loc}$ being the formal neighborhood of $Y$. Then $Y$ is decomposable if it is a transversal intersection of $Y_1\cap U_Y^{loc}$ and $Y_2\cap U_Y^{loc}$ and for any $\alpha \in \Phi^+$ such that $Y \subset \overline{T}_{\alpha}$ we have either $Y_1\subset \overline{T}_{\alpha} $ or $Y_2\subset \overline{T}_{\alpha} $. We call a layer $Y \in \cg^{\prime}$ indecomposable if it is locally indecomposable in the above sense.
It is known that the set $\mathcal{G}$ of all indecomposable $Y \in \cg^{\prime}$ forms {a \em building set}, see \cite[Definition 2.5]{cg}. 

\begin{rem}
Moci \cite{m} gives two different definitions of an irreducible layer. We use \emph{$\bc$-irreducible layers}, see  \cite[Section 3.1]{m}.  
\end{rem}

\begin{rem}
    In the definition of decomposition of $Y$, we require $Y$ to be a transversal intersection of $Y_1$ and $Y_2$ only locally, because in principle $Y$ could be one of possibly many connected components in the intersection of $Y_1$ and $Y_2$. The first example where this happens is $\Phi$ of the type $B_2$, and $Y_1, Y_2$ are codimension $1$ subtori in $T$ corresponding to long roots. Then their intersection consists of two points. One of them is the unity, in which all the root subtori meet (so it is indecomposable), while another is such $Y$ that decomposes into $Y_1$ and $Y_2$.
\end{rem}


We define $X_{\Phi}$ as the closure of $T^{reg}$ under the map:
$$T^{reg} \to \prod_{G \in \cg} Bl_G \overline{T}.$$
It is known that the variety $X_{\Phi}$ is smooth.
We also need the following description of $X_{\Phi}$ (see \cite{cg}):

\begin{thm}
If one arranges the elements $G_1,G_2,...,G_N$  of $\cg$ in the dimension increasing order, then \(X_{\Phi}\) is isomorphic to the variety  \[Bl_{{\widetilde G_N}}Bl_{{\widetilde G_{N-1}}}\
\cdots Bl_{{\widetilde G_2}}Bl_{G_1} \overline{T}\]
where \({\widetilde G_i}\) denotes the dominant transform of $G_i$ in $Bl_{{\widetilde G_{i-1}}}
\cdots Bl_{{\widetilde G_2}}Bl_{G_1} \overline{T}$.
\end{thm}

Similarly to the open cover of $\overline{T}$ by the sets $U_{\Delta,I}^{loc}$, we can cover the whole $X_\Phi$ by open subsets for all $\Delta$ and $I$:
$$ X_{\Delta,I}^{loc}:=X_\Phi\times_{\overline{T}}U_{\Delta,I}^{loc}.
$$
In particular, for any $\Delta$, we have $X_{\Delta,\Delta}^{loc}=X_\Phi^\circ$ and $X_{\Delta,\emptyset}^{loc}=U_{\Delta,\emptyset}^{loc}$.

\begin{prop}\label{pr:DCG-over-t-orbit}
    The isomorphism $U_{\Delta,I}^{loc}\simeq T_I\times E_{\Delta,I}^{loc}$ induces an isomorphism $X_{\Delta,I}^{loc}= X_{\Phi_I}^\circ\times E_{\Delta,I}^{loc}$. 
\end{prop}
\begin{proof}
    Follows from Proposition~\ref{pr:restriction-to-Levi} and the fact that blows-up commute with restrictions to open subsets.
\end{proof} 

This reduces the construction of charts on $X_\Phi$ to that for $X_\Phi^\circ$.

\subsection{Charts on $X_{\Phi}^{\circ}$}

For any layer $Y$, let $Y^\circ$ be the complement of all smaller layers in $Y$. Clearly, $Y^\circ$ is open in $Y$. Consider the open covering of $T$ by formal neighborhoods of $Y^\circ$ for all layers $Y$. Slightly abusing the notations, denote such formal neighborhood by  $U_Y^{loc}$. Since $T$ is smooth, each $U_Y^{loc}$ is the Cartesian product $U_Y^{loc}=Y^\circ\times E_Y$ where $E_Y$ is the formal neighborhood of $0$ in the transversal tangent space which is naturally identified with the Lie subalgebra $\mathfrak{h}_Y\subset\fh$ being the the intersection of $\fh$ with the semisimple part of the centralizer subalgebra $\mathfrak{z}_Y\subset\fg$ of any element of $Y^\circ$. 

Denote by $\Phi_Y$ the root system of $\fz_Y$. Clearly, it is a closed root subsystem in $\Phi$. Next, the intersection of the arrangement with $U_Y^{loc}$ is formed by the products $Y^\circ\times Z$ where $Z$ are elements of the root arrangement in $\mathfrak{h}_Y$ determined by $\mathfrak{z}_Y$. This means that we can cover $X_\Phi^\circ$ by open subsets $V_Y:=X_\Phi^\circ\times_{T}U_Y^{loc}$. 

\begin{prop}\label{pr:formal-neighborhood-of-layer}
    $V_Y$ is the Cartesian product $Y^\circ\times \widetilde{M}_{\Phi_Y}^{loc}$, where $\widetilde{M}_{\Phi_Y}^{loc}$ is the formal neighborhood of the central fiber ${M}_{\Phi_Y}$ in $\widetilde{M}_{\Phi_Y}$.
\end{prop}

\begin{proof}
 Since $X_\Phi^\circ$ is an iterated blow-up of the toric arrangement, this open subset is a similar blow-up of the intersection of this arrangement with $U_Y^{loc}$. The latter is formed by $Y^\circ\times Z$ where $Z$ are elements of the root arrangement in $\mathfrak{h}_Y$.   
\end{proof}

\begin{defn} For any maximal nested set $S$ in the root arrangement of $\fh_Y$, the open subset $V_{Y,S}$ is $Y^\circ\times \widetilde{W}_S^{loc}\subset Y^\circ\times \widetilde{M}_{\Phi_Y}^{loc}=V_Y$.
\end{defn}

The case of 0-dimensional layer $Y$ includes the situation when the root system $\Phi_Y$ is a full-rank maximal closed root subsystem in $\Phi$. In this case all the elements in the torus $T$ such that $\Phi_Y$ is the root system of their centralizer belong to the finite subgroup dual to the quotient $R/R_Y$ of the root lattices for $\Phi$ and $\Phi_Y$. Those are precisely $y \in T$ such that the centralizer $\fz_{\fg}(y)$ is a Borel de-Siebenthal subalgebra in $\fg$. For simple $\Phi$, such a quotient is always cyclic, of the order $1,2,3$, or $5$, moreover, for classical $\fg$ we always have $m=1$ or $2$. See \cite{BDS} for the complete classification of Borel--de Siebenthal subalgebras. So the above charts are in one-to-one correspondence with the following data:


    \begin{itemize}
        \item A maximal closed root subsystem $\Phi' \subset \Phi$, along with an isomorphism of $R / R_{\Phi'}$ with the  $ \mathbb{Z}/m\mathbb{Z}$ for an appropriate $m$. 
        \item A maximal nested set $S'$ corresponding to the root arrangement of hyperplanes in $\fh$  determined by $\Phi'$.
    \end{itemize}

\begin{rem}
    Not every $0$-dimensional layer corresponds to a Borel--de Siebenthal subalgebra. Indeed, for $\Phi$ of the type $C_n$, any splitting $n=k_1+\ldots+k_s$ determines a maximal rank subsystem $\Phi'\subset\Phi$ being the sum of root systems of types $C_{k_i}$. For $s>2$, this is a non-maximal subsystem of maximal rank. 
\end{rem}

We will use the fact that some particular rational functions on $T$ are regular on $V_{Y,S}$. Namely, a layer $Y$ and a nested set $S$ in $\mathcal{G}_Y$ determines a partial order on $\Phi_Y$: to any $\alpha\in\Phi_Y$, we assign $A_\alpha\in S$ being the maximal $A\in S$ such that $\alpha=0$ on $A$; then we have $\alpha>\beta$ iff $A_\alpha\supset A_\beta$. Then the following holds:

\begin{prop}\label{pr:Moci-regular-functions}
    The function $\dfrac{e^{\beta}-1}{e^{\alpha}-1}$ is regular on $V_{Y,S}$ if $\alpha >\beta$.
\end{prop}

\begin{proof}
    Indeed, under the isomorphism $V_{Y,S}=Y^\circ\times \widetilde{W}_S^{loc}$, the above function is the pullback of $\dfrac{\beta}{\alpha}$ on the second factor multiplied by a pullback of a regular function on $U_Y^{loc}$ that equals $1$ on $Y$. So the statement follows from Proposition~\ref{pr:DCP-reg-functions}.
\end{proof}

We can choose an \emph{adapted system of simple roots} $ \Delta_{Y,S}\subset\Phi_Y$ such that any element of $S$ is $A_\beta$ for some $\beta\in\Delta_{Y,S}$. Then we have the following
\begin{cor}\label{co:alpha-contains-beta}
    For $\alpha\in\Phi_Y$ and $\beta\in\Delta_{Y,S}$, the function $\dfrac{e^{\beta}-1}{e^{\alpha}-1}$ is regular on $V_{Y,S}$ if $\alpha$ contains $\beta$ in its decomposition with respect to $\Delta_{Y,S}$.
\end{cor}

\begin{proof}
  Indeed, for $\alpha\in\Phi_Y$ and $\beta\in\Delta_{Y,S}$ we have $\alpha>\beta$ if and only if $\alpha$ contains $\beta$ in its decomposition with respect to $\Delta_{Y,S}$.  
\end{proof}

\subsection{Charts on $X_{\Phi}$} 
Let $ \mathcal{S} $ be a quadruple $(\Delta,I,Y,S)$ where $\Delta$ is a system of simple roots in $\Phi$, $I\subset\Delta$, $Y$ is a layer of the root arrangement determined by $\Phi_I$ in $T_I$ (i.e. an irreducible component of intersection of all the subtori $T_{\alpha,I}\subset T_I$ for all $\alpha\in \Phi_I$), and $S$ is a nested set in the root arrangement of hyperplanes in $\fh_Y\subset\fh_I$, where $\fh_Y$ is the Cartan subalgebra of the semisimple part of the Lie algebra centralizer in $\fg_I$ of a generic point in $Y$.

From Proposition~\ref{pr:DCG-over-t-orbit} we have the open covering $X_{\Phi} = \bigcup V_{ \mS }$ where 
$$ V_{ \mS }:=(V_{S}\times E_{\Delta,I}^{loc})\times_{\overline{T}}U_{\Delta,I}^{loc}.
$$

The following is clear from the above and Proposition~\ref{pr:Moci-regular-functions}:

\begin{prop}
The following functions are regular on $V_{ \mS }$:
\begin{itemize}
    \item $u_\alpha:=e^{\alpha}$ for $\alpha\in \Phi^+$;
    \item $v_\beta:=\dfrac{1}{u_\beta-1}$ for $\beta\not\in \Phi_Y$;
    \item $v_{\alpha,\beta}:=\dfrac{{u_\alpha}-1}{u_{\beta}-1}$ for $\alpha\in\Phi_Y, \beta\not\in\Phi_Y$ or $\alpha,\beta\in\Phi_Y$ such that $\alpha <\beta$ under the partial order determined by $S$.
\end{itemize}
\end{prop}
    

\subsection{Specialization to fibers of $X_\Phi\to \overline{T}$}

Let $C\in T$ and $\Phi(C)\subset\Phi$ be the root system for the centralizer $\mathfrak{z}_\fg(C)$.

\begin{prop}
    We have $X_\Phi^\circ\times_{T}\{C\}=M_{\Phi(C)}$. Under this identification, we have $V_{\mS}\times_{T}\{C\}=W_S$ and the specialization of $v_{\alpha,\beta}$ to $W_S$ is $w_{\alpha,\beta}$.
\end{prop}

\begin{proof} There is a unique $Y^\circ$ containing $C$. Then $\fz_\fg(C)=\fz_Y$ and $\Phi(C)=\Phi_Y$. By Proposition~\ref{pr:formal-neighborhood-of-layer}, we have $X_\Phi^\circ\times_{T}U_Y^{loc}=Y^\circ\times \widetilde{M}^{loc}_{\Phi_Y}$, so $X_\Phi^\circ\times_{T}\{C\}=M_{\Phi_Y}=M_{\Phi(C)}$. The proof of the second statement is similar to that of Proposition~\ref{pr:Moci-regular-functions}. 

\end{proof}

Similarly, we have the following description of fibers of  $X_\Phi\to\overline{T}$. Suppose $y\in \overline{T}$ and let $O_{\Delta,I}$ be the $T$-orbit containing $y$. Then $y$ can be regarded as an element of the torus $T_I\subset L_I$. 

\begin{prop}\label{pr:gaudin-of-centralizer}
    Let $\Phi_Y$ be the root system of the centralizer of $y$ in $L_I$ and $S$ is a maximal nested set in the root arrangement of hyperplanes in $\fh_Y\subset\fh_I$. We have $X_\Phi\times_{\overline{T}}\{y\}=M_{\Phi_Y}$. Moreover, under this identification we have $V_ \mS \times_{\overline{T}}\{y\}=W_{S}$ and the specialization of $v_{\alpha,\beta}$ to $V_{ \mS }\times_{\overline{T}}\{y\}=W_{S}$ is $w_{\alpha,\beta}$.
\end{prop}

\begin{proof}
 The proof is entirely similar to the that of the previous proposition since the arrangement of subtori is transversal to the $T$-orbit $O_{\Delta,I}$ in $\overline{T}$, by Proposition~\ref{pr:restriction-to-Levi}.     
\end{proof}

\section{Commutative subspaces in the holonomy Lie algebra}

\subsection{$X_\Phi$ as the parameter space} 
Our goal is to extend the map
$$\psi: T^{reg} \to Gr(\rk \fg, \dim (\ft_{\Phi}^{trig})^{1})$$
to the charts $V_{ \mS } \subset X_{\Phi}$.
Let $\Delta_0$ be the set of simple roots in $\Phi^+$ and $\mS = (w(\Delta_0),I,Y,S)$ for some $w \in W$. 
Let $x \in V_\mS$. 
We choose a basis of $\fh$: \\
1) $b_{k+1}, \ldots, b_n$ -- the corresponding elements to the adapted basis  for a nested set $S$ in $\Phi_Y$;  \\
2) $b_{l+1}, \ldots, b_k$ is its orthogonal complement in $\spann \langle \Phi_I \rangle$; \\
3) $b_1, \ldots, b_l$  basis of $\spann \langle \Phi \setminus \Phi_{I} \rangle$ orthogonal to 
$$b_{l+1}, \ldots, b_k, b_{k+1}, \ldots, b_{n}.$$  
Let $\beta_i, i = k+1, \ldots, n$ the corresponding to $b_{k+1}, \ldots, b_n$ elements of $\fh^*$.

Then the following holds:

\begin{lem}\label{le:renormalized_hamiltonians} The following Hamiltonians are well defined in the chart $V_ \mS$:

$$BH_i(x) =  \delta(b_i) - \frac{1}{2}\sum_{\alpha \in \Phi^+\setminus \Phi_I^+} (u_\alpha(x) +1) v_\alpha(x) (\alpha, b_i) t_{\alpha} 
= \tau_w(b_i) - \sum_{\alpha \in \Phi^+ \setminus \Phi_I^+} u_\alpha(x) v_\alpha(x) (\alpha, b_i) t_{\alpha}$$
for $i=1,\ldots,l$, 

$$BH_i(x) =  \delta(b_i) - \frac{1}{2}\sum_{\alpha \in \Phi^+\setminus\Phi_Y^+} (u_\alpha(x) +1) v_\alpha(x) (\alpha, b_i) t_{\alpha} 
= \tau_w(b_i) - \sum_{\alpha \in \Phi^+ \setminus\Phi_Y^+} u_\alpha(x) v_\alpha(x) (\alpha, b_i) t_{\alpha}$$
for $i=l+1,\ldots,k$, and
\begin{eqnarray*}
BH_i(x) = \delta(b_i) (u_{\beta_i}(x)-1)
 - \frac{1}{2}\sum_{\alpha \in \Phi_Y^+}  (u_{\alpha}(x) +1)v_{\beta_i,\alpha}(x)(\alpha, b_i)  t_{\alpha} = \\ =\tau_w(b_i) (u_{\beta_i}(x)-1)
 - \sum_{\alpha \in \Phi_Y^+}  u_{\alpha}(x) v_{\beta_i,\alpha}(x)(\alpha, b_i)  t_{\alpha}
 \end{eqnarray*}
for $i=k+1,\ldots,n$.
\end{lem}

We denote by $Q(x)$ the subspace generated by the elements above. Note that $Q(x)$ is an abelian subspace because it is so for all $x=C\in T^{reg}$ and $T^{reg}$ is open in $V_{\mathcal{S}}$.
For $x=C\in T^{reg}$, the above elements are
$$BH(C, b_1), \ldots, BH(C, b_k);$$
$$ BH(C,b_{k+1}) \cdot (u_{\beta_{k+1}}(C)-1)  , \ldots, BH(C,b_n) \cdot  (u_{\beta_{n}}(C)-1).$$

Hence for $x=C \in T^{reg} \subset V_{\mathcal{S}}$ we have $Q(x)=Q(C)$.

\begin{prop}\label{pr:dimQ}
For any $x \in V_ \mS $, we have $\dim Q(x) = \rk \fg$. 
\end{prop}

\begin{proof}
Let us fix an element from $x \in V_ \mS $. 
For $i = 1, \ldots, k$ we have 
\begin{equation}\label{eq:1st-type-generators} BH_i(x) = \delta(b_i) - \frac{1}{2}\sum_{\alpha \in \Phi^+ \setminus \Phi_Y^+} (u_\alpha(x)+1) v_\alpha(x) (\alpha, b_i) t_{\alpha}.\end{equation}
For $i \in \{k+1, \ldots, n \}$ where 
we have
\begin{equation}\label{eq:3d-type-generators}
    BH_i(x) =  \sum_{\alpha \in \Phi_Y^+} v_{\beta_i \alpha}(x) \alpha(b_i)  t_{\alpha}.
\end{equation} 

Let  $y\in\overline{T}$ be the image of $x$. Then by Proposition~\ref{pr:gaudin-of-centralizer}, the preimage of $y$ in $V_ \mS $ gets identified with $W_{S}\subset M_{\Phi_Y}$. Let $z$ be the image of $x$ under this identification. Then, according to Proposition~\ref{pr:gaudin-of-centralizer}, the latter elements are
$$ BH_i(x) =  \sum_{\alpha \in \Phi_Y^+} w_{\beta_i \alpha}(z) \alpha(b_i)  t_{\alpha},$$
i.e. Gaudin elements in $\mathfrak{t}_{\Phi_Y}$ corresponding to $z$.

Consider the projection of $(\mathfrak{t}_\Phi^{trig})^{1}$ to $\delta(\fh)$ annihilating all the $t_\alpha$. This projection takes the first 
set of generators to linearly independent elements of $\delta(\fh)$. The third set of generators is annihilated by that projection and is linear independent by Proposition~\ref{pr:AFV}. So altogether, $BH_i(x)$ span the subspace of the dimension $n$ in $(\mathfrak{t}_\Phi^{trig})^{1}$, and this completes the proof.

\end{proof}

As a consequence we have a well-defined regular map $\psi:X_\Phi\to Gr(\rk\fg, \dim(\ft_\Phi^{trig})^1)$. Moreover, we have the following description of the subspaces $Q(x)$ for boundary points $x\in X_\Phi$. 

\begin{cor}
    Let $x\in X_\Phi$ be such that $y\in O_{\Delta,I}$ (with $\Delta=w(\Delta_0)$) is its image in $\overline{T}$, $Y\subset T_I$ is the layer such that $y\in Y^{\circ}$ and $z\in M_{\Phi_Y}$ be the element corresponding to $x$ under the identification of $X_\Phi\times_{\overline{T}}\{y\}$ with $M_{\Phi_Y}$. Then $Q(x)$ is the direct sum of the following three subspaces:
    \begin{itemize}
        \item $\tau(\fh_I^\perp)=\spann\{ \tau_w(h)\ |\ h\in\fh_I^{\perp}\}$;
        \item $\spann\{ \tau_w(h) - \sum_{\alpha \in \Phi_I^+ \setminus\Phi_Y^+} u_\alpha(x) v_\alpha(x) (\alpha, h) t_{\alpha}\ |\ h\in\fh_Y^{\perp}\cap\fh_I\}$, a subspace in $\ft_{\Phi_I}^{trig, 1}\subset\ft_\Phi^{trig,1}$;
        \item $\spann\{ \sum_{\alpha \in \Phi_Y^+}   w_{h,\alpha}(z)(\alpha, h)  t_{\alpha}\ |\ h\in\fh_Y\}$, i.e the Gaudin subspace in $\ft_{\Phi_Y}^{1}\subset\ft_\Phi^{trig, 1}$ corresponding to $z$.
    \end{itemize}
    
\end{cor}

\subsection{Injectivity of the map $\psi$.} We are going to show that the fibers of the above map $\psi$ are always finite and, moreover, it is injective in all classical cases. 

Let $\Gamma_Y\subset T_I$ be the finite group dual to $(R_{\Phi_I}\cap\mathbb{Q}\Phi_Y)/R_Y$, that is $\Gamma_Y=\{y\in T_I\ |\ \xi(y)=1\ \forall \xi\in R_Y^\perp\oplus R_Y\}$.

\begin{prop}
    The subspace $Q(x)$ uniquely determines the following data:
    \begin{itemize}
        \item  $\Phi_Y$;
        \item the image $y\in\overline{T}$ of $x\in X_\Phi$ up to $\Gamma_Y$;
        \item  $z\in M_{\Phi_Y}$. 
    \end{itemize}
\end{prop}

\begin{proof} Consider the kernel of projection of $Q(x)$ along span of $t_{\alpha}$ to $\delta(\fh)$. According to the proof of Proposition~\ref{pr:dimQ}, it is generated by 
$$ BH_i(x) =  \sum_{\alpha \in \Phi_Y^+} w_{\beta_i \alpha}(z) \alpha(b_i)  t_{\alpha},$$
for all $b_i$ in $\Delta_Y\subset\Phi_Y$. So this kernel space is the Gaudin subalgebra in $\mathfrak{t}_{\Phi_Y}$ corresponding to $z\in M_{\Phi_Y}$. In particular, both $\Phi_Y$ and $z$ are uniquely determined by $Q(x)$.

Next, the preimages of $\delta(h)$ under the above projection $Q(x)\to\delta(\fh)$ are nonzero for $h$ orthogonal to $\Phi_Y$ are uniquely determined up to the kernel of this projection, in particular, they are uniquely determined up to $\ft_{\Phi_Y}$. So the coefficients of such elements at $t_\alpha$ for all $\alpha\not\in\Phi_Y$ are uniquely determined by $Q(x)$. Since those coefficients are fractional-linear expressions of $u_\alpha(x)$, that are nonzero for $\alpha\in\Phi\setminus \Phi\cap\mathbb{Q}\Phi_Y$, this means that the values of $u_{\alpha}(x)\in\mathbb{CP}^1$ for $\alpha\in\Phi\setminus \Phi\cap\mathbb{Q}\Phi_Y$ are uniquely determined by $Q(x)$. Also, since $\Phi_Y$ is determined by $Q(x)$, we have $u_\alpha(x)=1$ for all $\alpha\in\Phi_Y$. Hence we can recover $y \in \overline{T}$ (at least) up to $\Gamma_Y$.

\end{proof}

\begin{cor}
\label{inj}
    Suppose that the root system $\Phi$ is of on of the classical types $A_n,B_n,C_n,D_n$. Then the map $\psi$ is injective, i.e. we have $Q(x_1) \ne Q(x_2)$ for any $x_1\ne x_2 \in X_{\Phi}$.
\end{cor}
\begin{proof}
Suppose that $Q(x_1) = Q(x_2)$. In this case $\Phi_Y$ is the same for $x_1$ and $x_2$ by the previous Proposition. Let $y_1, y_2$ be the corresponding points on $\overline{T}$ be different. They are defined up to $\Gamma_Y$ (which is also the same). 
If $\Phi_Y$ is maximal closed root subsystem of maximal rank then from classification of Borel-de-Siebental  it follows that $\Gamma_Y = \bz/2\bz$.
For arbitrary closed root subsystem $\Phi_Y$ from the classification in \cite[section 10.1]{oshima} it follows that any closed root subsystem is orthogonal sum of Levi subsystems in maximal closed root subsystem of maximal rank. This implies that the group $\Gamma_Y$ is necessary the product $\prod_i \bz/2\bz$. 

Since $y_1\ne y_2$, there exists 
$ \alpha \in (\Phi_I \cap \bq \Phi_Y) \setminus \Phi_Y$  such that $ u_{\alpha}(y_1) = 1, u_{\alpha}(y_2) = -1$. Then $\Phi_Y$ should be different for $y_1$ and $y_2$ and this is a contradiction which completes the proof.
\end{proof} 

\begin{example}
Suppose that $\Phi_Y \subset \Phi$ is a root system of a Borel-de-Siebental subalgebra $\fg^{\prime}$ of $\fg$: a point $x \in X_\Phi$ is a pair $(C, z)$, where $C \in T$ and $z \in M_{\Phi_I^{\prime}}$. Then $\fg^{\prime} = \fz_{\fg}(C)$. It is well-known fact that for different $C \in T$ Borel-de-Siebental subalgebras are different. This is a prototypical example for injectivity proof in Corollary \ref{inj}. 

For exceptional types in is not the case: Consider for example the root system $G_2$. The set of long roots form a closed root subsystem $A_2$ which is a centralizer of two different elements of $T$: one is given by $e^{\alpha} = \omega$ for any long root $\alpha$ and $e^{\beta} = 1$ for any short root, the second one is given by $e^{\alpha} = \omega^2$ for any long root $\alpha$ and $e^{\beta} = 1$ for any short root. Here $\omega^3 = 1, \omega \ne 1$. 
\end{example}


To summarize, we have proved the following theorem:

\begin{thm}\label{th:compactification-quadratic}
The map $\psi: T^{reg} \to Gr(\rk \fg, \dim t_{\Phi}^{1})$ uniquely extends to the map $\overline{\psi}: X_{\Phi} \to Gr(\rk \fg, \dim t_{\Phi}^{1})$. This map is injective for root systems $A_n,B_n,C_n,D_n$. 
\end{thm}
\begin{rem}
  Moreover for $\Phi = A_n$ from Proposition \ref{a_n_case} and Theorem \ref{afv12} it follows that $\overline{\psi(T^{reg})} = X_{\Phi}$.  
\end{rem}









\section{The vector bundle $\mathcal{Q}$ over $X_\Phi$}

In this section we globalize the construction of the commutative subspaces
$Q(x)\subset (\ft^{trig}_\Phi)^{1}$ and prove that they form a vector bundle
over $X_\Phi$, which can be naturally identified with the tangent
bundle of $X_\Phi$ logarithmic along the boundary divisor
$D:=X_\Phi\setminus T^{reg}$.

\subsection{The Bethe bundle}
Recall from Theorem~\ref{th:compactification-quadratic} that the map
\[
\psi\colon T^{reg}\longrightarrow Gr\big(n,(\ft^{trig}_\Phi)^{(1)}\big)
\]
extends uniquely to a regular map
\[
\overline{\psi}\colon X_\Phi\longrightarrow Gr\big(n,(\ft^{trig}_\Phi)^{(1)}\big).
\]

Let $\mathcal{T}$ denote the tautological rank $n$ subbundle in the trivial bundle $(\ft^{trig}_\Phi)^{1}\otimes\mathcal{O}_{Gr(n,(\ft^{trig}_\Phi)^{1})}$ over
$Gr(n,(\ft^{trig}_\Phi)^{1})$. Here $\mathcal{O}_X$ stands for the structure sheaf of the variety $X$. We set:

\begin{defn}
The \emph{Bethe bundle} $\mathcal{Q}$ over $X_\Phi$ is the pullback
of $\mathcal{T}$ along $\overline{\psi}$:
\[
\mathcal{Q}:=\overline{\psi}^*(\mathcal{T})\subset
(\ft^{trig}_\Phi)^{1}\otimes\mathcal{O}_{X_\Phi}.
\]
\end{defn}

By construction, for every point $x\in X_\Phi$, the fiber $\mathcal{Q}_x$ is
precisely the commutative subspace $Q(x)\subset (\ft^{trig}_\Phi)^{1}$
constructed in the previous section. In particular, $\mathcal{Q}$ is a vector
bundle of rank $n$ and restricts over $T^{reg}$ to the
subbundle spanned by Bethe Hamiltonians.

\subsection{The trigonometric Casimir connection and logarithmic singularities}

Let $\{\alpha_i\}_{i=1}^r$ be the simple roots and $\{\omega_i^\vee\}_{i=1}^r$ be the corresponding fundamental coweights and set
$u_i=e^{\alpha_i}$, so that $(u_1,\dots,u_r)$ are multiplicative coordinates
on the maximal torus $T$. In these coordinates, the trigonometric Casimir
connection can be written 
\begin{equation}
\label{eq:trig-casimir}
\nabla = \kappa\, d
+\frac{1}{2}\sum_{\alpha\in\Phi^+}
t_\alpha\,
\frac{e^\alpha+1}{e^\alpha-1}\,\frac{de^\alpha}{e^\alpha}
-\sum_{i=1}^r \delta(\omega_i^\vee)\,\frac{du_i}{u_i}
=\kappa\,d-\sum_{i=1}^r BH(\omega_i^\vee)\,\frac{du_i}{u_i},
\end{equation}
where $\{h_i\}_{i=1}^r\subset\fh$ is the basis dual to $\{\omega_i\}$ and
$BH(h_i)$ are the Bethe Hamiltonians as in Section~\ref{lie_alg}. We denote
by $\omega$ the $(\ft^{trig}_\Phi)^{1}$–valued 1–form on $T$ given by
\[
\omega=\nabla-\kappa d
=\frac{1}{2}\sum_{\alpha\in\Phi^+}
t_\alpha\,
\frac{e^\alpha+1}{e^\alpha-1}\,\frac{de^\alpha}{e^\alpha}
-\sum_{i=1}^r \delta(\omega_i^\vee)\,\frac{du_i}{u_i}
=-\sum_{i=1}^r BH(\omega_i^\vee)\,\frac{du_i}{u_i}.
\]

\begin{lem}
\label{lem:log-extension}
The form $\omega$ extends from $T$ to a $(\ft^{trig}_\Phi)^{1}$–valued
1–form on $X_\Phi$ with logarithmic singularities along the boundary divisor
$D=X_\Phi\setminus T^{reg}$. Equivalently, $\omega$ defines a global section
\[
\omega\in H^0\Bigl(X_\Phi,
\;\Omega^1_{X_\Phi}(\log D)\otimes_{\mathcal{O}_{X_\Phi}}
(\ft^{trig}_\Phi)^{1}\otimes\mathcal{O}_{X_\Phi}\Bigr).
\]
\end{lem}

\begin{proof}
First rewrite each summand in~\eqref{eq:trig-casimir} as a linear combination
of logarithmic differentials. Note that
\[
\frac{e^\alpha+1}{e^\alpha-1}\,\frac{de^\alpha}{e^\alpha}
= 2\,d\log(e^\alpha-1)-d\log(e^\alpha).
\]
Thus $\omega$ is a linear combination of forms of the type
$d\log(e^\alpha)$ and $d\log(e^\alpha-1)$, $\alpha\in\Phi^+$, with
coefficients in $(\ft^{trig}_\Phi)^{1}$. Over the torus $T$, these are
logarithmic 1–forms with the poles along the divisors $e^\alpha=0$,
$e^\alpha=\infty$ and $e^\alpha=1$.

The variety $X_\Phi$ is obtained from the toric variety $\overline{T}$ by a
sequence of blow-ups along smooth centers contained in the union of the
divisors $\overline{T}_\alpha=\{e^\alpha=1\}$ and the boundary of
$\overline{T}$. In particular, the pullback of the $1$-form $\omega$ to $X_\Phi$ still locally has the form of a linear combination of $\frac{df_j}{f_j}$ for some meromorphic functions $f_j$ with possible zeros on the boundary divisor (i.e. on the full preimage of the arrangement of subtori and the boundary of $\overline{T}$).
Hence $\omega$ extends to a global section of
$\Omega^1_{X_\Phi}(\log D)\otimes_{\mathcal{O}_{X_\Phi}}((\ft^{trig}_\Phi)^{1}\otimes\mathcal{O}_{X_\Phi}))$,
as claimed.
\end{proof}

\subsection{The pairing with the logarithmic tangent bundle}

Since the fiber of $\mathcal{Q}$ at $x\in X_\Phi$ is the commutative
subspace $Q(x)\subset(\ft^{trig}_\Phi)^{1}$ spanned by the coefficients of $\omega$, the form $\omega$ is a global section
\[
\omega\in H^0\bigl(X_\Phi,\;
\Omega^1_{X_\Phi}(\log D)\otimes\mathcal{Q}\bigr),
\]
i.e. it induces a bundle morphism
\begin{equation}
\label{eq:bundle-pairing}
\Theta\colon
T_{X_\Phi}(-\log D)\;\longrightarrow\;\mathcal{Q}
\end{equation}
obtained by contraction with $\omega$. Concretely, at a point
$x\in X_\Phi$ and for $v\in T_x X_\Phi(-\log D)$, we have $\Theta_x(v)= \iota_v\omega$.

\begin{thm}\label{th:bundle}
    $\mathcal{Q}$ is the tangent bundle logarithmic at the boundary divisor $D=X_\Phi\setminus T^{reg}$.
\end{thm}
\begin{proof}
We have the morphism~\eqref{eq:bundle-pairing} of vector bundles
\[
\Theta\colon T_{X_\Phi}(-\log D)\;\longrightarrow\;\mathcal{Q}
\]
Both $T_{X_\Phi}(-\log D)$ and $\mathcal{Q}$ are locally free of rank
$n=\mathrm{rk}\,\fg$ on the smooth variety $X_\Phi$. We have to show that the above morphism is an isomorphism on any open set $V_{\mathcal{S}}\subset X_\Phi$. 
On a given $V_{\mathcal{S}}$, we have the local coordinates given by $u_i$ for $i=1,\ldots,k$ and $v_{\beta_i,\gamma_i}$ for $i=k+1,\ldots,n$, with $\gamma_i$ being the elements preceding $\beta_i$ in the partial order on the set simple roots of $\Phi_Y$ determined by $S$. Their logarithmic differentials trivialize the bundle $\Omega^1_{V_{\mathcal{S}}}(\log D)$. Note that the renormalized Hamiltonians from  Lemma~\ref{le:renormalized_hamiltonians} are the coefficients of $\dfrac{du_i}{u_i}$ for $i=1,\ldots,k$ and $\dfrac{d(u_{\beta_i}-1)}{u_{\beta_i}-1}$ for $i=k+1,\ldots,n$. On the other hand, we have $u_{\beta_i}-1=u_{\beta_i^s}\prod\limits_{j=1}^{s-1} v_{\beta_i^j,\beta_i^{j+1}}$ where $\beta_i^1=\beta_i,\ \beta_i^2=\gamma_i$ and $\beta_i^{j+1}$ precedes $\beta_i^j$ in the partial order on $\Delta_Y$ determined by the nested set $S$. So the logarithmic differentials of $u_{\beta_i}$ and of $v_{\beta_i,\gamma_i}$ are related by a triangular linear transformation. Hence the renormalized Hamiltonians from  Lemma~\ref{le:renormalized_hamiltonians} are obtained by a triangular linear transformations from the coefficients of $\omega$ in the basis of sections $\Omega^1_{V_{\mathcal{S}}}(\log D)$ given by $\frac{du_i}{u_i}$ for $i=1,\ldots,k$ and $\frac{dv_{\beta_i,\gamma_i}}{v_{\beta_i,\gamma_i}}$ for $i=k+1,\ldots,n$. So $\Theta$ is an isomorphism on a fiber over any $x\in X_\Phi$, so it is an isomorphism of vector bundles on
all of $X_\Phi$. Equivalently, we obtain an isomorphism
\[
\mathcal{Q}\;\simeq\;T_{X_\Phi}(-\log D),
\]
which is precisely the statement of the theorem.
\end{proof}

\subsection{Restriction to the fiber over $1\in T$} Theorem~\ref{th:bundle} has the following consequence describing the restriction of $\mathcal{Q}$ to the central fiber $M_\Phi=X_\Phi\times_{\overline{T}}1$. The restriction of $\mathcal{Q}$ to $M_\Phi$ is sheaf $\mathcal{G}$ of Gaudin subalgebras in $\ft_\Phi$ from \cite{afv2}. We can describe this sheaf entirely in terms of $M_\Phi$, namely, we have the following generalization of \cite[Theorem~3.3]{afv}

\begin{cor}
The morphism 
\[
p : \widetilde{M}_\Phi \longrightarrow M_\Phi
\]
is a line bundle, denote it by $L$. Then the sheaf \(\mathcal{G}\) of Gaudin subalgebras is a vector bundle on \(M_\Phi\) isomorphic, as a locally free sheaf, to
\[
\mathcal{D}^{1}_{L^\vee}(-\log D),
\]
the sheaf of first order differential operators on \(M_\Phi\) twisted by \(L^\vee\) whose symbol is logarithmic along \(D\). In particular, there is an exact sequence of sheaves on \(M_\Phi\):
\[
0 \longrightarrow \mathcal{O}_{M_\Phi} \longrightarrow \mathcal{G} \longrightarrow T_{M_\Phi}(-\log D) \longrightarrow 0,
\]
where the embedding of the trivial line bundle \(\mathcal{O}_{M_\Phi}\) sends \(1\) to
\[
c_\Phi = \sum\limits_{\alpha\in\Phi^+} t_{\alpha}.
\]
\end{cor}

\begin{proof}
    The formal neighborhood of $M_\Phi$ in $X_\Phi$ is isomorphic to the formal neighborhood $M_\Phi^{loc}$ of $M_\Phi$ in $\overline{M}_\Phi$, and by Theorem~\ref{th:bundle} the restriction of $\mathcal{Q}$ to $M_\Phi^{loc}$ is $T_{M_\Phi^{loc}}(-\log D)$. So the first statement $\mathcal{G}=\mathcal{D}^{1}_{L^\vee}(-\log D)$ follows.

    Next, the image of \(\mathcal{O}_{M_\Phi}\) is spanned by the pairing of the $1$-form $\omega$ with the Euler vector field on $M_\Phi^{loc}$, and that is $c_\Phi$.
\end{proof}

\subsection{Sheaves of commutative subspaces}

Let $U(\ft^{trig}_\Phi)$ be the universal enveloping algebra and let
\[
\varphi : U(\ft^{trig}_\Phi)\longrightarrow A
\]
be a homomorphism of unital associative algebras (for example,
$A=\End(V)$ for a representation $V$ of $\ft^{trig}_\Phi$).

Composing the inclusion
$\mathcal{Q}\hookrightarrow \ft^{trig}_\Phi\otimes\mathcal{O}_{X_\Phi}$ with
$\varphi$ we obtain an $\mathcal{O}_{X_\Phi}$–linear map
\[
\omega_A := (\varphi\otimes\Id_{\mathcal{O}_{X_\Phi}})\circ\omega\colon
\mathcal{Q} \longrightarrow A\otimes \mathcal{O}_{X_\Phi}.
\]
Extending $\omega_A$ multiplicatively to the symmetric algebra gives a
homomorphism of sheaves of commutative algebras
\[
\varphi : S\mathcal{Q}\longrightarrow A\otimes \mathcal{O}_{X_\Phi}.
\]
We denote by $\mathcal{E}_A$ the image sheaf; its local sections form
commutative subspaces of $A$, and $\mathcal{E}_A$ is a sheaf of
$\mathcal{O}_{X_\Phi}$–algebras on~${X_\Phi}$.

\begin{cor}
\label{cor:Poisson-kernel}
Let $\varphi : U(\ft^{trig}_\Phi)\to A$ be an algebra homomorphism.  
Then the kernel of the induced homomorphism of sheaves of algebras
\[
\varphi : S\mathcal{Q} \longrightarrow A\otimes \mathcal{O}_{X_\Phi}
\]
is an ideal in $S\mathcal{Q}$ closed with respect to the Poisson bracket.
\end{cor}

\begin{proof}
The argument is parallel to that of Corollary~4.1 in \cite{afv}.
Since $S\mathcal{Q}$ is the symmetric algebra of the sheaf of Lie algebras
$\mathcal{Q}=T_{X_{\Phi}}(-\log D)$, its Poisson bracket is
determined by the Lie bracket on sections of $\mathcal{Q}$ via the
Leibniz rule.

Let $\xi_1,\dots,\xi_k,\eta_1,\dots,\eta_\ell$ be local sections of
$\mathcal{Q}$.  Then
\[
\{\xi_1\cdots \xi_k,\eta_1\cdots \eta_\ell\}
=
\sum_{i,j}
[\xi_i,\eta_j]\,
\xi_1\cdots \widehat{\xi_i}\cdots\xi_k
\eta_1\cdots \widehat{\eta_j}\cdots\eta_\ell.
\]
Applying $\varphi$ we obtain
\begin{align*}
\varphi\bigl(\{\xi_1\cdots \xi_k,\eta_1\cdots \eta_\ell\}\bigr)
&=
\sum_{i,j}
\omega_A([\xi_i,\eta_j])\,
\prod_{r\neq i}\omega_A(\xi_r)\,
\prod_{s\neq j}\omega_A(\eta_s).
\end{align*}
Since all $\omega_A(\xi_r)$ and $\omega_A(\eta_s)$ lie in the commutative
image sheaf $\mathcal{E}_A$, they commute pairwise.

Rewriting as in \cite{afv} one finds
\[
\varphi\bigl(\{\xi_1\cdots \xi_k,\eta_1\cdots \eta_\ell\}\bigr)
=
\sum_i
\Bigl(\xi_i\varphi(\eta_1\cdots\eta_\ell)\Bigr)
\prod_{r\neq i}\omega_A(\xi_r)
-
\sum_j
\Bigl(\eta_j\varphi(\xi_1\cdots\xi_k)\Bigr)
\prod_{s\neq j}\omega_A(\eta_s).
\]
Since such products generate $S\mathcal{Q}$, this shows that if
$a,b\in S\mathcal{Q}$ satisfy $\varphi(a)=\varphi(b)=0$, then
$\varphi(\{a,b\})=0$.  Hence $\ker\varphi$ isclosed with respect to the Poisson bracket.
\end{proof}

\subsection{Coisotropic and Lagrangian spectra}

Via the identification
\[
S\mathcal{Q}\;\simeq\;\mathcal{O}\bigl(T^*X_{\Phi}(\log D)\bigr),
\]
we have a homomorphism
\[
\varphi : S\mathcal{Q}\longrightarrow A\otimes\mathcal{O}_{X_{\Phi}}.
\]

Let $\mathcal{E}_A$ be the image sheaf, and set
\[
Z_A := \Spec_{X_{\Phi}}(\mathcal{E}_A),
\]
the relative spectrum over $X_{\Phi}$.  By construction, $Z_A$ is a closed
subscheme of $T^*X_{\Phi}(\log D)$, finite over $X_{\Phi}$ when $A$ is
finite dimensional.

\begin{cor}
\label{cor:coisotropic-spectrum}
Let $\varphi : U(\ft^{trig}_\Phi)\to A$ be an algebra homomorphism.
Let $\mathcal{E}_A$ be the corresponding sheaf of commutative
$\mathcal{O}_{{X_{\Phi}}}$–algebras and $Z_A := \Spec_{{X_{\Phi}}}(\mathcal{E}_A)$ its relative
spectrum.

\begin{enumerate}
    \item $Z_A$ is a coisotropic subscheme of the Poisson variety
    $T^*X_{\Phi}(\log D)$.

    \item If $A$ is finite dimensional, then the projection $Z_A\to X_{\Phi}$ is a finite morphism, and   $Z_A|_{T^{reg}}$ is a Lagrangian
    subvariety of the symplectic manifold
    $T^*T^{reg} \subset T^*X_{\Phi}(\log D)$.
\end{enumerate}
\end{cor}

\begin{proof}
(1) By Corollary~\ref{cor:Poisson-kernel}, the kernel of the surjection
$S\mathcal{Q}\twoheadrightarrow\mathcal{E}_A$ is an ideal closed with respect to the  Poisson bracket.  Under the
identification $S\mathcal{Q}\simeq\mathcal{O}(T^*X_{\Phi}(\log D))$,
such ideals correspond exactly to defining ideals of coisotropic subschemes.
Thus $Z_A$ is coisotropic in $T^*X_{\Phi}(\log D)$.

(2) If $A$ is finite-dimensional then, for any $x\in X_\Phi$, the joint spectrum of elements of  $\mathcal{Q}_x$ regarded as operators of left multiplication on $A$ is finite, so the projection $Z_A\to X_{\Phi}$ is a finite morphism. In particular, the corresponding subscheme
$Z_A|_{T^{reg}}$ has dimension equal to $\dim T$, which is half of
$\dim T^*T^{reg}$.  Being coisotropic and of half dimension, 
$Z_A|_{T^{reg}}$ is Lagrangian.
\end{proof}
















\noindent\footnotesize{
{\bf Aleksei Ilin} \\
Higher School of Modern Mathematics, MIPT, Russia \\
HSE University, Moscow, Russia \\
{\tt alex.omsk2@gmail.com}}\\

\noindent\footnotesize{
{\bf Leonid Rybnikov} \\
Department of Mathematics and Statistics,
University of Montreal, Montreal QC, Canada\\
{\tt leo.rybnikov@gmail.com}}


\begin{thebibliography}{99}

\bibitem[AFV1]{afv}
L. Aguirre, G. Felder, A. Veselov,
{\em Gaudin subalgebras and stable rational curves.}
Compositio Mathematica, {\bf 147} (2011), no.~5, pp.~1463--1478.

\bibitem[AFV2]{afv2}
L. Aguirre, G. Felder, A. Veselov,
{\em Gaudin subalgebras and wonderful models.}
Selecta Mathematica (N.S.), {\bf 22} (2016), no.~3, pp.~1057--1071.

\bibitem[ATL]{at}
A. Appel, V. Toledano Laredo,
{\em Pure braid group actions on category $\mathcal{O}$ modules.}
Pure and Applied Mathematics Quarterly, {\bf 20} (2024), no.~1, pp.~29--79.

\bibitem[BDS]{BDS}
A. Borel, J. De Siebenthal,
{\em Les sous-groupes fermés de rang maximum des groupes de Lie clos.}
Commentarii Mathematici Helvetici, {\bf 23} (1949), pp.~200--221.

\bibitem[BMO]{bmo10}
A. Braverman, D. Maulik, A. Okounkov,
{\em Quantum cohomology of the Springer resolution.}
Advances in Mathematics, {\bf 227} (2011), no.~1, pp.~421--455.

\bibitem[D]{drin}
Vladimir G. Drinfeld,
{\em Hopf algebras and the quantum Yang--Baxter equation},
Soviet Mathematics Doklady, {\bf 32} (1985), no.~1, 254--258.

\bibitem[DCG]{cg}
C. De Concini, G. Gaiffi,
{\em Projective Wonderful Models for Toric Arrangements.}
Advances in Mathematics, {\bf 327} (2018), pp.~390--409.

\bibitem[DCP1]{DCP}
C. De Concini, C. Procesi,
{\em Wonderful models of subspace arrangements.}
Selecta Mathematica (N.S.), {\bf 1} (1995), no.~3, pp.~459--494.

\bibitem[DCP2]{DCP2}
C. De Concini, C. Procesi,
{\em Hyperplane Arrangements and Holonomy Equations.}
Selecta Mathematica (N.S.), {\bf 1} (1995), no.~3, pp.~495--535.

\bibitem[G1]{g1} M. Gaudin, Diagonalisation d'une classe d'Hamiltoniens de spin, \textit{J. Physique} \textbf{37} (1976), no.10, 1089--1098.

\bibitem[G2]{g2} M. Gaudin, La fonction d’onde de Bethe, \textit{Collect. Commissariat \'Energ. Atom. S\'er. Sci.} Masson, Paris, 1983. 

\bibitem[HKRW]{hkrw}
I. Halacheva, J. Kamnitzer, L. Rybnikov, A. Weekes,
{\em Crystals and monodromy of Bethe vectors.}
Duke Mathematical Journal, {\bf 169} (2020), no.~12, pp.~2337--2419.

\bibitem[HLLY]{hlly}
I.~Halacheva, A.~Licata, I.~Losev, and O.~Yacobi,
{\em Categorical braid group actions and cactus groups},
Advances in Mathematics, {\bf 429} (2023), 109190,
doi:10.1016/j.aim.2023.109190.

\bibitem[I]{ilin}
A. Ilin,
{\em The maximality of certain commutative subalgebras in Yangians.}
Functional Analysis and Its Applications, {\bf 53} (2019), no.~4, pp.~309--312.

\bibitem[IKLPR]{iklpr}
A.~Ilin, J.~Kamnitzer, Y.~Li, P.~Przytycki, and L.~Rybnikov,
{\em The moduli space of cactus flower curves and the virtual cactus group},
arXiv:2308.06880 (2023).

\bibitem[IKR]{ikr}
A. Ilin, J. Kamnitzer, L. Rybnikov,
{\em Gaudin models and moduli space of flower curves.}
arXiv:2407.06424.

\bibitem[IMR]{imr}
A. Ilin, I. Mashanova-Golikova, L. Rybnikov,
{\em Spectra of Bethe subalgebras of $Y(\mathfrak{gl}_n)$ in tame representations.}
Letters in Mathematical Physics, {\bf 112} (2022), no.~5, Article~99.

\bibitem[IR]{ir}
A. Ilin, L. Rybnikov,
{\em Degeneration of Bethe subalgebras in the Yangian of $\mathfrak{gl}_n$.}
Letters in Mathematical Physics, {\bf 108} (2018), no.~4, pp.~1083--1107.

\bibitem[IR2]{ir2}
A.~Ilin and L.~Rybnikov,
{\em Bethe subalgebras in Yangians and the wonderful compactification},
Communications in Mathematical Physics, {\bf 372} (2019), no.~2, pp.~343–366.
doi:10.1007/s00220-019-03509-1.

\bibitem[Ka]{K}
M. Kapranov,
{\em Chow quotients of Grassmannians I.}
Journal of Algebraic Geometry, {\bf 4} (1995), no.~3, pp.~533--545. 
(arXiv:alg-geom/9210002v1.)

\bibitem[Ko]{k}
T. Kohno,
{\em Integrable connections related to Manin and Schechtman's higher braid groups.}
Illinois Journal of Mathematics, {\bf 34} (1990), no.~2, pp.~476--484.

\bibitem[KMR]{kmr}
V.~Krylov, I.~Mashanova-Golikova, and L.~Rybnikov,
{\em Bethe subalgebras in Yangians and Kirillov–Reshetikhin crystals},
arXiv:2212.11995 (2022).

\bibitem[KR]{kr}
J.~Kamnitzer and L.~Rybnikov,
{\em Cactus flower spaces and monodromy of Bethe vectors},
arXiv:2507.12829 (2025).

\bibitem[KTWWY]{ktwwy}
J. Kamnitzer, P. Tingley, B. Webster, A. Weekes, O. Yacobi,
{\em Highest weights for truncated shifted Yangians and product monomial crystals.}
Compositio Mathematica, {\bf 155} (2019), no.~9, pp.~1617--1652. 
(arXiv:1511.09131v2.)

\bibitem[L]{l}
E. Looijenga,
{\em Arrangements, KZ systems and Lie algebra homology.}
In: Singularity Theory (Liverpool, 1996), 
London Mathematical Society Lecture Note Series {\bf 263}, 
Cambridge University Press, 1999, pp.~109--130.

\bibitem[LM]{LM}
A. Losev, Y. Manin,
{\em New moduli spaces of pointed curves and pencils of flat connections.}
Michigan Mathematical Journal, {\bf 48} (2000), no.~1, pp.~443--472.

\bibitem[M]{m}
L. Moci,
{\em Wonderful models for toric arrangements.}
International Mathematics Research Notices, {\bf 2012} (2012), no.~8, pp.~213--238.

\bibitem[MO]{mo} D. Maulik, A. Okounkov {\em Quantum Groups and Quantum Cohomology.} arXiv:1211.1287

\bibitem[O]{oshima}
T. Oshima,
{\em A classification of subsystems of a root system.}
Publications of the Research Institute for Mathematical Sciences, Kyoto University,
{\bf 12} (1976), no.~2, pp.~387--438.

\bibitem[P]{jp}
J. Peters, {\em Compactifying the Parameter Space for the Quantum
Multiplication for Hypertoric Varieties.} 	arXiv:2510.15687

\bibitem[T]{tl}
V. Toledano Laredo,
{\em The trigonometric Casimir connection of a simple Lie algebra.}
Journal of Algebra, {\bf 329} (2011), pp.~286--327.

\bibitem[W]{wend}
C. Wendlandt,
{\em The $R$-matrix presentation for the Yangian of a simple Lie algebra.}
Communications in Mathematical Physics, {\bf 363} (2018), no.~1, pp.~289--332.

\end{thebibliography}
\end{document}